\documentclass[10pt,reqno]{amsart}

\usepackage[a4paper]{geometry}%,text={16cm,23cm}]{geometry}
\usepackage{enumerate,enumitem}
\usepackage[utf8]{inputenc}
\usepackage[T1]{fontenc} 
\usepackage[english]{babel}
\usepackage{array}
\usepackage{amsmath,amssymb,amsthm}
\usepackage{xfrac}
\usepackage{xcolor}
\usepackage{graphicx}
\usepackage{mathrsfs}
\usepackage{url}
\usepackage{hyperref}
\hypersetup{colorlinks=true,breaklinks=true,urlcolor=blue,linkcolor=blue,bookmarksopen=true} 
\usepackage{enumitem}   

\usepackage{newtxmath}
\newcommand{\be} {\begin{equation}}
\newcommand{\ee} {\end{equation}}
\newcommand{\bea} {\begin{eqnarray}}
\newcommand{\eea} {\end{eqnarray}}
\newcommand{\Bea} {\begin{eqnarray*}}
\newcommand{\Eea} {\end{eqnarray*}}

\newcommand*{\X}{\mathcal{X}}
\newcommand*{\M}{\mathcal{M}}
\newcommand*{\A}{\mathscr A}
\renewcommand*{\L}{\mathscr L}
\renewcommand{\P}{\mathcal P}
\newcommand{\D}{\mathcal D}

\newcommand{\1}{\mathbf 1}
\newcommand*{\N}{\mathbb{N}}

\newcommand*{\R}{\mathbb{R}}
\newcommand*{\TV}{\mathrm{TV}}
\newcommand*{\ba}{\mathbf{a}}
\newcommand*{\bQ}{\mathbf{Q}}

\newcommand*{\e}{\mathrm{e}}

\DeclareMathOperator{\supp}{supp}

\DeclareMathOperator{\sign}{sign}
\DeclareMathOperator{\Ent}{Ent}
\DeclareMathOperator*{\esssup}{ess\,sup}

\newtheorem{theo}{Theorem}[section]
\newtheorem{prop}[theo]{Proposition}

\newtheorem{lem}[theo]{Lemma}

\newtheorem{Rq}[theo]{Remark}

\numberwithin{equation}{section}

\definecolor{darkred}{rgb}{0.9,0.1,0.1}

\begin{document}

\title[Replicator-mutator house of cards model]{Fast, slow convergence, and concentration\\
 in the house of cards replicator-mutator model}

\author{
	Bertrand \textsc{Cloez}
		\and
	Pierre \textsc{Gabriel}
}
\def\runauthor{
	Bertrand \textsc{Cloez} \and Pierre \textsc{Gabriel}
}

\date{\today}
    \address[B. \textsc{Cloez}]{MISTEA, INRAE, Intitut Agro, Univ. Montpellier, 2 place Pierre Viala, 34060 Montpellier, France.}
    \email{bertrand.cloez@inrae.fr}
    \address[P. \textsc{Gabriel}]{Laboratoire de Math\'ematiques de Versailles, UVSQ, CNRS, Universit\'e Paris-Saclay,  45 Avenue des \'Etats-Unis, 78035 Versailles cedex, France.}
    \email{pierre.gabriel@uvsq.fr}

\keywords{evolutionary genetics; selection-mutation; house of cards model; long time behavior; concentration phenomenon}  

\subjclass[2010]{45K05, 45M05, 92D15}

\begin{abstract}
We propose a fine analysis of the various possible long time behaviours of the solutions of the replicator-mutator equation with so-called Kingman's house of cards mutations.
In particular, we give what is to our knowledge the first concentration result for this model.
\tableofcontents
\end{abstract}

\maketitle

\section{Introduction}

We are interested in the long time behavior of the following non-linear integro-differential equation
\begin{equation}\label{eq:hoc}
\partial_t v_t(x) + a(x) v_t(x) = Q(x) \int_\X v_t(y) dy + v_t(x) \int_\X \big(a(y)-1\big) v_t(y) dy
\end{equation}
defined for $t>0$ and $x\in\X$, a measurable subset of $\R^d$, and complemented with an initial condition $v_0(x)$.
The function $Q(x)$ is strictly positive probability density function on~$\X$, and the function $a(x)$ is bounded below and continuous on~$\X$.

\medskip

This equation is a particular case of Kimura's replicator-mutator model in evolutionary biology where $\X$ stands for a set of phenotypic traits~\cite{K65},
and we refer to~\cite{CFM08,WYY17} for a rigorous derivation from individual based models.
The quantity $-a(x)$ represents the fitness of the phenotype $x$, namely the difference between the birth and death rates. 
In the general model, the mutation term is $\int_\X K(x,y)v(t,y)dy$, where $K(x,y)$ represents the creation rate of individuals with trait $x$ from individual with trait $y$.
The particular case $K(x,y)=Q(x)$, where the $x$ distribution is the same whatever the original trait $y$, is known as the \emph{house of cards} model of mutations after the work of Kingman~\cite{K78}, see~\cite{Burgerbook,B88,BB96,CC07,CCDR16,T84}.
In \cite{K78}, Kingman neglects the small mutations and only takes into account the mutations that have a significant effect on the population. The latter, often deleterious, destroy the biochemical `house of cards' created by evolution.
We are interested here in this model which, despite its apparent simplicity, captures the main features of the general case and in particular the possible concentration phenomenon, see below.
The quadratic term in Equation~\eqref{eq:hoc} can be seen as a Lagrange multiplier ensuring that, if initially a probability distribution $v_0(x)$, then for all $t>0$ the solution $v_t(x)$ is a probability distribution which represents the relative frequency of the traits in the population.
A consequence of this conservativeness property is that the probability density solutions are insensitive to the addition of a constant to the fitness function $a(x)$.
We will thus assume that this lower bounded function is actually nonnegative.
This emphasizes that Equation~\eqref{eq:hoc} does not belong to the class of logistic type selection-mutation models which appear in evolutionary ecology.
In these models, the competition for resources leads to negative quadratic terms, see for instance~\cite{BMP09,CC04,CCDR13,DJMP05,GVA06,Magal02bis,Magal02,MW00,MPW14}.

\medskip

Our aim is to give a precise description of the long time behavior of Equation~\eqref{eq:hoc} in the case when the fitness $a(x)$ reaches its minimum -- which can be assumed to be zero -- at a unique point which, up to a translation of $\X$, can be assumed without loss of generality to be the origin.
We thus make the following hypotheses on $\X$, $a$, and $Q$:
\begin{itemize}[parsep=1mm]
\item[{\bf(H$\X$)}] The trait space $\X$ is a measurable subset of $\R^d$ which contains a neighborhood of $0$.
\item[{\bf(H$Q$)}] The mutation kernel $Q:\X\to(0,\infty)$ is a probability density function.
\item[{\bf(H$a$)}\,] The fitness $a:\X\to[0,\infty)$ is a continuous function which satisfies
\[a(0)=0,\quad a(x)>0\ \text{for all}\ x\in\X\setminus\{0\},\quad\text{and}\quad\int_{\X\cap\{|x|>1\}}\frac{Q(x)}{a^2(x)}dx<\infty.\]%\liminf_{|x|\to\infty}a(x)>0.\]
\end{itemize}
The last condition on $a$ prevents the escape to infinity, a phenomenon that can occur in the replicator-mutator model~\cite{AC14,AC17,TLK96}.

\medskip

Equation~\eqref{eq:hoc} is strongly related to the non-conservative linear equation
\begin{equation}\label{eq:lin-noncons}
\partial_t u_t(x) = \A u_t(x) = - a(x) u_t(x) + Q(x) \int_\X u_t(y) dy.
\end{equation}
For any nonnegative and non identically zero solution $u_t(x)$ of Equation~\eqref{eq:lin-noncons}, the function
\[v_t(x)=\frac{u_t(x)}{\int_\X u_t(y)\,dy}\]
is solution to Equation~\eqref{eq:hoc}.
Reciprocally, if $v(t,x)$ satisfies Equation~\eqref{eq:hoc}, then the function
\[u_t(x)=v_t(x)\,\exp\Big(\int_0^t\int_\X \big(1-a(y)\big)v_s(y)dyds\Big)\]
verifies Equation~\eqref{eq:lin-noncons}.
Also, finding a stationary probability distribution for Equation~\eqref{eq:hoc} is equivalent to find $\lambda\in\R$ and a probability measure $\gamma$ on $\X$ such that $\A\gamma=\lambda\gamma$.
This Perron-Frobenius eigenproblem can be easily solved and the explicit expression of the unique solution depends on whether the parameter
\[\rho=\int_\X\frac{Q(x)}{a(x)}\,dx,\]
which can be infinite, is larger or smaller than $1$.
More precisely, if $\rho\geq1$ then $\gamma$ has a Lebesgue density given by
\begin{equation}\label{eq:gamma-rho>1}
\gamma(dx)=\frac{Q(x)}{\lambda+a(x)}dx,
\end{equation}
where $\lambda\geq0$ is the unique real number such that $\int_\X\frac{Q}{\lambda+a}=1$.
If $\rho<1$, then $\lambda=0$ and $\gamma$ has an atom at zero:
\begin{equation}\label{eq:gamma-rho<1}
\gamma(dx)=(1-\rho)\delta_0+\frac{Q(x)}{a(x)}dx.
\end{equation}
These computations are made in~\cite{BB96,Co10,Co13,K78} and the presence of a Dirac mass in the case $\rho<1$ suggests a concentration phenomenon, supported by numerical evidences~\cite{BCL17}, similarly as in~\cite{K78}.
However, to the best of our knowledge, there is no proof of convergence to a singular measure for the solutions of Equation~\eqref{eq:hoc} in the literature.
Providing such a result is the main purpose of the present paper, but we also prove new convergence estimates in the non-singular case $\rho\geq1$.

\medskip

For characterizing the long time behavior of the solutions, the dual Perron eigenvalue problem is helpful.
It consists in finding a non-negative and non-zero function $h$ such that $\A^*h=\lambda h$, where the dual operator $\A^*$ is given by
\begin{equation}\label{eq:A*}
\A^*f(x)=-a(x)f(x)+\int_\X f(y)Q(y)\,dy.
\end{equation}
Similarly as for the direct problem, the solutions can be computed explicitly.
In the case $\rho>1$ they are positive and given by
\begin{equation}\label{eq:h-rho>1}
h(x)=\frac{\alpha^{-1}}{\lambda+a(x)}
\end{equation}
with $\alpha$ any positive constant.
The convenient choice we make is to take $\alpha=\int_\X\frac{Q}{(\lambda+a)^2}$ so that $h$ is the unique eigenfunction such that $\langle\gamma,h\rangle=\int_\X h(x)\gamma(dx)=1$.
In the case $\rho<1$ there is no strictly positive eigenfunction but only degenerated ones which are zero everywhere and positive at zero.
Due to the singularity of $\gamma$ in this case, we can nevertheless define a unique such eigenfunction~$h$ such that $\langle\gamma,h\rangle=1$ by setting
\begin{equation}\label{eq:h-rho<1}
h(x)=\left\{\begin{array}{cl}
\frac{1}{1-\rho}&\text{if}\ x=0,\vspace{1mm}\\
0&\text{otherwise.}
\end{array}\right.
\end{equation}
In the critical case $\rho=1$, both the functions $\frac1a$ and $\1_{\{0\}}$ are formally eigenfunctions of $\A^*$.
Yet, the relevance of these eigenfunctions depends on the choice of the Banach space we consider.
If studying Equation~\eqref{eq:lin-noncons} in the $L^1$ Lebesgue space associated to the measure $\frac{dx}{a(x)}$, then the eigenfunction $\1_{\{0\}}$ corresponds to the zero linear form and the only non-trivial eigenfunction is~$\frac{1}{a}$,
which can be normalized to have $\langle\gamma,h\rangle=1$ provided that $\int_\X\frac{Q}{a^2}<\infty$.
In contrast, if we work in a space of measures that contains the Dirac mass $\delta_0$, then $\frac{1}{a}$ is not a bounded linear form on this space and $\1_{\{0\}}$ is the relevant eigenfunction.

\medskip

We work in various Banach spaces that we recall here.
For a positive weight function $\varphi$ on $\X$ and $p\in[1,\infty)$ we denote by $L^p(\varphi)=L^p(\X,\varphi(x)dx)$ the standard Lebesgue space associated to the measure $\varphi(x)dx$.
When $\varphi=\1$, the constant function equal to $1$, we use the shorthand $L^p$ for $L^p(\1)$ and $L^\infty$ for $L^\infty(\X,dx)$ endowed with the norm $\left\|f\right\|_{L^\infty}=\esssup_\X|f|$.
For any $p\in[1,\infty]$, we use the standard notation $p'\in[1,\infty]$ for the Hölder conjugate exponent, {\it i.e.} such that $1/p+1/p'=1$.
%We denote by $\mathcal L^\infty(\X)$ the space of bounded Borel measurable functions on $\X$ equipped with the supremum norm $\|f\|_\infty=\sup_\X|f|$.
We denote by $\M(\X)$ the space of finite signed measures on $\X$ endowed with the total variation norm $\|\mu\|_\TV=\mu_+(\X)+\mu_-(\X)=\sup_{\|f\|_\infty\leq1}\langle\mu,f\rangle$,
where the supremum is taken over the measurable or continuous functions $f:\X\to\R$ such that $\|f\|_\infty=\sup_\X|f|\leq1$ and the duality bracket is given by $\langle\mu,f\rangle:=\int_\X f\,d\mu$.
Due to the canonical injection $L^1\hookrightarrow\M(\X)$, we will often abuse notations and identify $L^1$ functions to their associated Lebesgue density measure,
and in particular for $\mu\in L^1$, we will use the notation $d\mu$ for the measure $\mu(dx)=\mu(x)dx$.
We denote by $\P(\X)\subset\M(\X)$ the subset of probability measures,
and for a positive weight function $\varphi$ we denote by $\M(\varphi)$ the space of signed measures $\mu$ such that $\|\mu\|_{\M(\varphi)}=\langle\mu_+,\varphi\rangle+\langle\mu_-,\varphi\rangle=\sup_{\|f/\varphi\|_\infty\leq1}\langle\mu,f\rangle<\infty$.
We also recall that vague convergence of measures means convergence for any test function in $C_c(\X)$ while narrow convergence means convergence for any test function in $C_b(\X)$.

\medskip

Our main results are summarized in the two following theorems.
The first one is about the linear equation~\eqref{eq:lin-noncons}.

\begin{theo}
\label{th:main-lin-noncons}
We suppose that {\bf(H$\X$)-(H$Q$)-(H$a$)} are met.
Then the following results hold:
\begin{enumerate}[leftmargin=8mm,itemsep=4mm]
\item \textbf{Exponential ergodicity.} \label{th:main-lin-noncons-fast}
Assume that $\rho=\int_\X Q/a\in(1,+\infty]$.
Then $\lambda>0$, the eigenvectors $\gamma$ and $h$ are given by~\eqref{eq:gamma-rho>1} and~\eqref{eq:h-rho>1} respectively, with $\alpha$ such that $\langle\gamma,h\rangle=1$, and
\vspace{2mm}
\begin{enumerate}[leftmargin=6mm]
\item\label{th:main-lin-noncons-fast-TV} For all $u_0\in\M(h)$ and all $t\geq0$,
\[\left\| \e^{-\lambda t}u_t - \langle u_0,h\rangle \gamma \right\|_{\M(h)} \leq \e^{-\lambda t}\left\| u_0 - \langle u_0,h\rangle \gamma \right\|_{\M(h)}.\]
\item\label{th:main-lin-noncons-fast-Lr_1<r<2} For all $p\in[1,2]$, all $u_0\in L^p(\gamma^{1-p}h)$, and all $t\geq 0$,
 \[ \left\| \e^{-\lambda t}u_t - \langle u_0,h\rangle \gamma \right\|_{L^p(\gamma^{1-p}h)} \leq \e^{-\lambda t}\left\| u_0 - \langle u_0,h\rangle \gamma \right\|_{L^p(\gamma^{1-p}h)}.\]
\item\label{th:main-lin-noncons-fast-Lr_r>2} For all $p>2$, all $u_0\in L^p(\gamma^{1-p}h)$, and all $t\geq 0$,
 \[ \left\| \e^{-\lambda t}u_t - \langle u_0,h\rangle \gamma \right\|_{L^p(\gamma^{1-p}h)} \leq 2\,\e^{-\lambda t}\left\| u_0 - \langle u_0,h\rangle \gamma \right\|_{L^p(\gamma^{1-p}h)}.\]
\item\label{th:main-lin-noncons-fast-Linf} For all $u_0\in L^1(h)$ such that $u_0/\gamma\in L^\infty$ and all $t\geq 0$,
 \[ \left\|\e^{-\lambda t}u_t /\gamma - \langle u_0,h \rangle \right\|_{L^\infty} \leq 2\,\e^{-\lambda t}\left\| u_0/\gamma - \langle u_0,h \rangle \right\|_{L^\infty}.\]
\end{enumerate}

\item \label{th:main-lin-noncons-slow} \textbf{Slow convergence and unboundedness.} Assume that $\rho=1$ and $1/a\in L^2(Q)$.
Then $\lambda=0$ and, for $\gamma$ and $h$ given by~\eqref{eq:gamma-rho>1} and~\eqref{eq:h-rho>1} with $\langle\gamma,h\rangle=1$, we have
\vspace{2mm}
\begin{enumerate}[leftmargin=6mm]
\item\label{th:main-lin-noncons-slow-weak} For all $u_0\in\M(h)$,
\[u_t \xrightarrow[t\to+\infty]{}\langle u_0,h\rangle \gamma\]
in the vague topology, and in the narrow topology if additionally $a\in L^\infty$.
\item[(a')\!]\label{th:main-lin-noncons-linear-growth}
For $u_0=\delta_0$ we have for all $f\in C_c(\X)$, and for all $f\in C_b(\X)$ if additionally $a\in L^\infty$,
\[\langle u_t,f\rangle \sim \left\|1/a\right\|^{-2}_{L^2(Q)}\langle\gamma,f\rangle\, t\qquad\text{as}\ t\to+\infty.\]
\item\label{th:main-lin-noncons-slow-TV} If $1/a\in L^{1+q}(Q)$ for some $q>1$ and $a\in L^\infty$, then there exists $C>0$ such that for all $u_0\in\M((\1+a^{-q})h)$ and all $t\geq0$,
\[\left\| u_t - \langle u_0,h\rangle \gamma \right\|_{\M(h)}\leq C t^{1-q}\left\|u_0\right\|_{\M((\1+a^{-q})h)}.\]
\item \label{th:main-lin-noncons-slow-inf} If $1/a\in L^{1+q}(Q)$ for some $q>1$, then there exists $C>0$ such that for any $u_0\in L^1(h)$ such that $u_0/\gamma\in L^\infty$, any $p\in[1,\infty)$, and all $t\geq 0$,
\[\left\| u_t - \langle u_0,h\rangle \gamma \right\|_{L^p(\gamma^{1-p}h)}\leq C t^{-\frac{q-1}{p}}\left\| u_0/\gamma- \langle u_0,h\rangle \right\|_{L^\infty}.\]
\item \label{th:main-lin-noncons-slow-r} If $1/a\in L^{1+q}(Q)$ for some $q>1$ and, for $r>p$, either $p\in[1,2]$ and $a\in L^{2+r'}(Q)$, or $p\geq2$ and $a\in L^{2+p}(Q)$, then there exists $C>0$ such that for all $u_0\in L^r(\gamma^{r-1}h)$ and all $t\geq 0$,
$$ 
 \left\| u_t - \langle u_0,h\rangle \gamma \right\|_{L^p(\gamma^{1-p}h)} \leq  C t^{-(q-1)\left(1/p -1/r\right) } \left\| u_0 - \langle u_0,h\rangle \right\|_{L^{r}(\gamma^{1-r}h)}.
$$
\end{enumerate}

\item \label{th:main-lin-noncons-weak} \textbf{Degenerate convergence.}
\vspace{2mm}
\begin{enumerate}[leftmargin=6mm]
\item \label{th:main-lin-noncons-weak-1/a} If $\rho=1$ and $1/a\not\in L^2(Q)$, then for all $u_0\in \M(1/a)$, the convergence
\[u_t\xrightarrow[t\to+\infty]{}0\]
holds in the vague topology, and in the narrow topology if additionally $a\in L^\infty$.
\item \label{th:main-lin-noncons-weak-1} If $\rho<1$, then $\lambda=0$, the eigenvectors $\gamma$ and $h$ are given by~\eqref{eq:gamma-rho<1} and~\eqref{eq:h-rho<1} respectively, and for all $u_0\in \M(\X)$ we have in the narrow topology
\[u_t\xrightarrow[t\to+\infty]{} \langle u_0,h\rangle\gamma = u_0(\{0\})\left(\delta_0+\frac{1}{1-\rho} \frac{Q}{a}\right).\]
\end{enumerate}
\end{enumerate}
\end{theo}

\medskip

It is worth precising that for the definition of $\M(1/a)$ in the above theorem, we have set $(1/a)(0)=+\infty$.
In particular a measure in $\M(1/a)$ has no atom at zero.
The second main theorem is about the replicator-mutator equation~\eqref{eq:hoc}.

\begin{theo}
\label{th:main-nonlin}
Under Hypotheses {\bf(H$\X$)-(H$Q$)-(H$a$)}, the following results hold:
\vspace{2mm}
\begin{enumerate}[leftmargin=8mm,itemsep=4mm]
\item \textbf{Fast convergence.} \label{th:main-nonlin-fast}
Assume that $\rho=\int_\X Q/a\in(1,+\infty]$.
Then $\lambda>0$, the eigenvectors $\gamma$ and $h$ are given by~\eqref{eq:gamma-rho>1} and~\eqref{eq:h-rho>1}, with $\langle\gamma,h\rangle=1$, and
\vspace{2mm}
\begin{enumerate}[leftmargin=6mm]
\item\label{th:main-nonlin-fast-TV} If $a\in L^\infty$, then for any $v_0\in\P(\X)\cap\M(h)$ there exists $C>0$ such that for all $t\geq0$,
\[\left\| v_t - \gamma \right\|_{\M(h)} \leq C\e^{-\lambda t}.\]
\item\label{th:main-nonlin-fast-Lr} If $p\in[1,2)$ and $a\in L^{p'-2}(Q)$, or if $p\geq2$, then for any $v_0\in \P(\X)\cap L^p(\gamma^{1-p}h)$ there exists $C>0$ such that for all $t\geq 0$,
 \[ \left\| v_t - \gamma \right\|_{L^p(\gamma^{1-p}h)} \leq C\e^{-\lambda t}.\]
\item\label{th:main-nonlin-fast-Linf} For all $v_0\in \P(\X)$ such that $v_0/\gamma\in L^\infty$ there exists $C>0$ such that for all $t\geq 0$,
 \[ \left\|v_t /\gamma - \1 \right\|_{L^\infty} \leq C\e^{-\lambda t}.\]
\end{enumerate}

\item \label{th:main-nonlin-slow} \textbf{Slow convergence.} Assume that $\rho=1$.
Then $\lambda=0$ and, for $\gamma$ and $h$ given by~\eqref{eq:gamma-rho>1} and~\eqref{eq:h-rho>1} with $\langle\gamma,h\rangle=1$, we have
\vspace{2mm}
\begin{enumerate}[leftmargin=6mm]
\item\label{th:main-nonlin-slow-weak} For any $v_0\in\P(\X)$, we have in the narrow topology
\[v_t \xrightarrow[t\to+\infty]{}\gamma.\]
\item\label{th:main-nonlin-slow-TV} If $1/a\in L^{1+q}(Q)$ for some $q>1$ and $a\in L^\infty$, then for any $v_0\in\P(\X)\cap\M((\1+a^{-q})h)$ there exists $C>0$ such that for all $t\geq0$,
\[\left\| v_t - \gamma \right\|_{\M(h)}\leq C t^{1-q}.\]
\item \label{th:main-nonlin-slow-inf} If $1/a\in L^{1+q}(Q)$ for some $q>1$ and either $p\in[1,2)$ with $a\in L^{p'-2}(Q)$, or $p\geq2$, then for any $v_0\in \P(\X)$ such that $v_0/\gamma\in L^\infty$, there exists $C>0$ such that for all $t\geq 0$,
\[\left\| v_t - \gamma \right\|_{L^p(\gamma^{1-p}h)}\leq C t^{-\frac{q-1}{p}}.\]
\item \label{th:main-nonlin-slow-r} If $1/a\in L^{1+q}(Q)$ for some $q>1$ and, for $r>p$, either $p\in[1,2)$ and $a\in L^{2+r'}(Q)\cap L^{p'-2}(Q)$, or $p\geq2$ and $a\in L^{2+p}(Q)$, then for all $v_0\in \P(\X)\cap L^r(\gamma^{r-1}h)$ there exists $C>0$ such that for all $t\geq 0$,
$$ 
 \left\| v_t - \gamma \right\|_{L^p(\gamma^{1-p}h)} \leq  C t^{-(q-1)\left(1/p -1/r\right) }.
$$
\end{enumerate}

\item \label{th:main-nonlin-conc} \textbf{Concentration.}
\vspace{2mm}
If $\rho<1$ and $\lim_{|x|\to\infty}a(x)=+\infty$, then for all $v_0\in \P(\X)$
\[\frac1t\int_0^tv_s\,ds\xrightarrow[t\to+\infty]{}\gamma=(1-\rho)\delta_0+\frac{Q}{a}\]
in the narrow topology.
\end{enumerate}

\end{theo}

\medskip

To our knowledge, these results are new in the literature.
The concentration phenomenon for Equation~\eqref{eq:hoc} in the case $\rho<1$, which is proved here to occur in Cesàro mean, was expected since the works of Kingman~\cite{K78}, Bürger and Bomze~\cite{BB96} and more recently Coville {\it et al.}~\cite{BCL17,Co13}, but never rigorously established until now.
Concentration phenomena are very relevant in evolutionary biology.
It was proved to occur for logistic type nonlinearities but only in the pure selection case~\cite{ACT16,AFT05,AMHF99,DJMR08,JR11,LP20,Perthame,Raoul12} or in the vanishing mutation regime~\cite{BMP09,DjidjouDemasse2017,LMP11,PB08,Raoul11}.
For Kimura's replicator-mutator equation, a result similar to ours is proved in~\cite{GHMR17,GHMR19} for the specific case $\X=\R$ with $a(x)=x$ and purely deleterious convolutive mutations, {\it i.e.} $K(x,y)=J(x-y)$ with $\supp J\subset(-\infty,0]$, by means of an explicit formulation of the solutions through Laplace transform.
This approach cannot be adapted to house of cards mutations and our proof rather uses the explicit expression of the stationary distribution.

The convergence of the solutions of Equation~\eqref{eq:hoc} in the cases $\rho>1$ and $\rho=1$ with $1/a\in L^2(Q)$ are consequences of the convergence results of the linear equation~\eqref{eq:lin-noncons}.
This scheme of proof has for instance been used in~\cite{DM15,O21}.
Exponential convergences are deduced from functional inequalities for suitable entropies, in the spirit of~\cite{toulouse,B94,BGL,CGR10,Ch04}.
The question of quantifying the spectral gap of positive semigroups is a difficult question in general, see~\cite{Khaladi2000}.
Functional inequalities are an efficient tool for tackling this problem, see for instance~\cite{AlfGabKav} where such a quantified inequality is proved for the replicator-mutator model with convolutive mutations.
Here, due to the simplicity of the mutation kernel, this method allows us to derive optimal rates of convergence.
Polynomial convergences also rely on the use of entropies, but with weaker and somewhat more original functional inequalities; see however the closely related approaches \cite{BDLS21,CM18,KMN21,RW01}.
In the cases $\rho=1$ with $1/a\not\in L^2(Q)$ and $\rho<1$, the results on the linear equation are not enough for deriving the convergence of the nonlinear equation.
We then work directly on Equation~\eqref{eq:hoc} and prove fine upper and lower bounds of the solutions when times goes to infinity.

Our results extend those of Kingman~\cite{K78} to the time continuous setting and to general trait spaces $\X$ and coefficients $a$ and $Q$.
The convergence in the case $\rho\geq1$ corresponds to what Kingman calls democracy;
 the effect of the selection is simply to modify the shape of the distribution.
The concentration in the case $\rho<1$ is named meritocracy by Kingman;
the Dirac mass emerges from the growth of a new class of highly fitted individuals, and the smooth part consists of the descendants of these mutants.
Yet another regime is considered by Kingman, which is the case where $\supp Q\subset\{x\in\X,\ a(x)\geq\epsilon\}$ for some $\epsilon>0$, and $\supp v_0\cap\{x\in\X,\ a(x)<\epsilon\}\neq\emptyset$.
In this situation, called aristocracy, some initial individuals are always better adapted than the mutants.
In our study we do not consider this non-irreducible case, where the long time behavior strongly depends on $v_0$.

\medskip

The paper is organized as follows.
We first analyse in Section~\ref{sec:lin-cons} a conservative linear equation which is closely related to Equation~\eqref{eq:lin-noncons}.
Then we use the results of Section~\ref{sec:lin-cons} to prove Theorem~\ref{th:main-lin-noncons} in Section~\ref{sec:noncons}.
Finally, in Section~\ref{sec:nonlin}, we prove Theorem~\ref{th:main-nonlin} by taking advantage of the results in Theorem~\ref{th:main-lin-noncons}.

\section{A related conservative equation}
\label{sec:lin-cons}

In this section, we focus on the closely related and simpler conservative equation given by
\begin{equation}
\label{eq:cons}
\partial_tu_t(x)=\L u_t(x)=-\ba(x)u_t(x)+\bigg(\int_\X \ba(y)u_t(y)dy\bigg) \bQ(x),\qquad x\in \X,
\end{equation}
where $\bQ(x)dx$ is a probability measure and $\ba$ is a continuous nonnegative function such that
\begin{equation}
\label{as:anonnul}
\forall x\in\X\setminus\{0\},\ \ba(x)>\ba(0)\geq0.
\end{equation}
It is worth noticing that, unlike Equation~\eqref{eq:lin-noncons} for which adding a constant to $a$ only translates the spectrum to the left or to the right, adding a constant to $\ba$ really modifies the equation.
We consequently do not assume  that $\ba(0)=0$.
At infinity, we impose the integrability condition
\begin{equation}\label{as:bQ/ba}
\int_{\X\cap\{|x|>1\} }\frac{\bQ(x)}{\ba(x)}\,dx<\infty.
\end{equation}

\medskip

Equation~\eqref{eq:cons} is a pure mutation equation, {\it i.e.} a pure jump process, without birth nor death.
It is thus conservative in the sense that the integral of the solutions is preserved along time.
We can prove rigorously the well-posedness of the equation and its conservativeness by using the theory of strongly continuous semigroups.

The operator $\L$ with dense domain $\M(\1+\ba)\subset\M(\X)$ is closed, dissipative, and positive resolvent.
%\[\lambda\|\mu\|=\|\lambda|\mu|\|\leq\|(\lambda-\L)|\mu|\|+\|\L|\mu|\|=\|(\lambda-\L)|\mu|\|\]
Invoking the Lumer-Phillips theorem, see for instance~\cite{EN}, we deduce that it generates a positive strongly continuous contraction semigroup.
Since $\int_\X\L\mu=0$ for any $\mu\in\M(\1+\ba)$, this semigroup is even stochastic, meaning that it leaves invariant the set $\P(\X)$ of probability measures.
We call this semigroup $(P_t)_{t\geq0}$, where $P_t$ maps $\mu\in\M(\X)$ to $\mu P_t\in\M(\X)$.
It yields the solutions to Equation~\eqref{eq:cons} in the sense that
\[\frac{d}{dt}\mu P_t=(\L\mu)P_t=\L(\mu P_t),\qquad\text{for all}\ \mu\in\M(\1+\ba),\]
and
\[\mu P_t=\mu+\L\int_0^t\mu P_s\,ds,\qquad\text{for all}\ \mu\in\M(\X).\]
%It is a standard result by considering $\L$ as a $\ba$-compact perturbation of the operator $u\mapsto-\ba u$, which generates an explicit analytic semigroup, see for instance~\cite[Chapter III]{EN}.
Besides, considering $\L$ as a $\ba$-bounded perturbation of the operator $u\mapsto-\ba u$, which generates an explicit contraction semigroup, we have that $(P_t)_{t\geq0}$ satisfies the Duhamel formula
\[\mu P_t= \e^{-t\ba}\mu+\int_0^t\langle\mu,\ba \e^{-s\ba}\rangle \bQ P_{t-s}\,ds,\qquad\text{for all}\ \mu\in\M(\1+\ba).\]
Denoting by $\mathcal L^\infty(\1+\ba)$ the space of measurable Borel functions $f$ on $\X$ such that $\|f\|_{\mathcal L^\infty(\1+\ba)}=\sup_\X|f/(\1+\ba)|<\infty$,
we can define by duality a right  action of the semigroup $(P_t)_{t\geq0}$ on $\mathcal L^\infty(\1+\ba)$ by setting
\[P_tf(x)=\langle\delta_xP_t,f\rangle\]
for all $x\in\X$.
Since $\delta_x\in\M(\1+\ba)$, we deduce from the properties of the left action that
\[\partial_tP_tf(x)=\L^* P_tf(x)=P_t\L^* f(x)\]
for all $x\in\X$, which ensures in particular that $\varphi_t=P_tf$ is solution to the dual equation
\begin{equation}\label{eq:dual-cons}
\partial_t\varphi_t(x)=\L^*\varphi_t(x)=\ba(x) \left( \int_\X \varphi_t(y) \bQ(dy) - \varphi_t(x) \right)
\end{equation}
with $\varphi_0=f$, and also the Duhamel formula
\begin{equation}\label{eq:Duhamel-cons}
P_tf(x)=f(x)\e^{-\ba(x)t}+\int_0^t\ba(x)\e^{-\ba(x)s}\langle \bQ,P_{t-s}f\rangle\,ds.
\end{equation}
Since $\L^*\1=\1$, the subspace $\mathcal L^\infty(\X)$ of bounded measurable Borel functions on $\X$ is left invariant under the right action of $(P_t)_{t\geq0}$.
The restriction of $(P_t)_{t\geq0}$ to this subspace endowed with the supremum norm is a Markov semigroup, namely a positive contraction semigroup with the property that $P_t\1=\1$.
We point out that $(P_t)_{t\geq0}$ is strongly continuous on $\mathcal L^\infty(\1+\ba)$ and on $\mathcal L^\infty(\X)$ only in the case when $\ba$ is bounded.

Arguing similarly as for the operator $\L$, we also have that the operator $\L^*$ with dense domain $L^1(\bQ)\cap L^1(\bQ/\ba)\subset L^1(\bQ/\ba)$ generates a positive strongly continuous contraction semigroup on $L^1(\bQ/\ba)$.
This semigroup, that we still denote by $(P_t)_{t\geq0}$, verifies the Duhamel formula~\eqref{eq:Duhamel-cons} for all $f\in L^1(\bQ)\cap L^1(\bQ/\ba)$.
Moreover we have $(\L^*f)\bQ/\ba=\L(f\bQ/\ba)$ for all $f\in L^1(\bQ)\cap L^1(\bQ/\ba)$, and consequently
\begin{equation}\label{eq:dualityPt}
(P_tf)\frac\bQ\ba=\Big(f\frac\bQ\ba\Big)P_t.
\end{equation}
This relation between the left and right actions of $P_t$, that will be useful in our study, ensures in particular that if $\mu\in\M(\X)$ has a density with respect to Lebesgue's measure, then so does $\mu P_t$.
When $\bQ/\ba$ is integrable, we assume without loss of generality, by rescaling time, that
\[\pi(dx)=\frac{\bQ(x)}{\ba(x)}dx\]
is a probability measure.
It is easily seen that $\L\pi=0$ and $\pi$ is thus an invariant measure of $(P_t)_{t\geq0}$.

\medskip

We now introduce the so-called $\Phi$-entropies, see for instance~\cite{Ch04}.
For a convex function $\Phi : \mathcal I\to \mathbb{R}$, where $\mathcal I\subset\R$, a function $f: \X \to \mathcal I $ and a probability measure $\mu$, the $\Phi$-entropy associated to $\mu$ is defined by
$$
\Ent^{\Phi}_\mu(f)= \int \Phi(f) d\mu - \Phi\left( \int f d\mu \right)=\langle\mu,\Phi(f)\rangle-\Phi(\langle\mu,f\rangle)
$$
and is nonnegative due to Jensen's inequality.
Besides, Jensen's inequality also guarantees that for all $t\geq s\geq0$ and all $x\in\X$,
\[\Phi(P_tf)(x)=\Phi(\langle\delta_xP_{t-s}, P_sf\rangle) \leq \langle\delta_xP_{t-s},\Phi(P_sf)\rangle=P_{t-s} \Phi(P_s f)(x),\]
and consequently, since $\pi P_{t-s}=\pi$,
$$
\Ent^{\Phi}_\pi(P_tf) \leq \Ent^{\Phi}_\pi(P_sf).
$$
In other words, any $\Phi$-entropy associated to $\pi$ decreases along the solutions of Equation~\eqref{eq:dual-cons}.
In particular, taking $\Phi(x)=|x|^p$ with $p\geq1$, we get that $L^p(\pi^{1-p})$ is invariant under the right action of $(P_t)_{t\geq0}$.
If we consider $\Phi(x)=x\log x$, we get that for any positive functions $\mu\in\P(X)\cap L^1$, the Kullback-Leibler divergence $D_{\textrm{KL}}(\mu P_t \|\, \pi)$ from $\mu P_t$ to $\pi$, where $D_{\textrm{KL}}(f \| g)$ is defined by
$$
D_{\textrm{KL}}(f \| g) = \int_\X f(x) \log\left(\frac{f(x)}{g(x)} \right)dx,
$$
is non-increasing along time.
Indeed, from~\eqref{eq:dualityPt} we have $\mu P_t=P_t\big(\frac{\mu}{\pi}\big)\pi$ and then $D_{\textrm{KL}}(\mu P_t || \pi)=\Ent^{x\mapsto x\log x}_\pi\!\big(P_t\big(\frac{\mu}{\pi}\big)\big)$.
Note that $D_{KL}$ is not a distance since it is not symmetric and it does not satisfy the triangle inequality, but it controls the $L^1$ distance due to Pinsker's inequality
\[\left\|f-g\right\|_{L^1}\leq\sqrt{2D_{\textrm{KL}}(f \| g)},\]
see for instance~\cite[Lemma~2.5, p.88]{Tsybakov2009}.

\medskip

When $\ba(0)>0$, we can divide the equation $\L\pi=0$ by $\ba$ and we get that $\pi$ is the unique invariant probability measure of $(P_t)_{t\geq0}$.
In contrast, when $\ba(0)=0$, the Dirac mass $\delta_0$ is also invariant.
In the case where $\bQ/\ba$ is not integrable, we can no longer normalize $\pi$ to be a probability measure, but due to Assumptions~\eqref{as:anonnul} and~\eqref{as:bQ/ba} and the continuity of $\ba$, we necessarily have $\ba(0)=0$ and $\delta_0$ is the unique invariant probability measure.
We thus have three distinct situations:
\vspace{1mm}
\begin{itemize}[label=$\Diamond$,leftmargin=8mm,itemsep=1.5mm]
\item if $\int_\X\bQ/\ba=1$ and $\ba(0)>0$, then $\pi$ is the unique invariant probability measure;
\item if $\int_\X\bQ/\ba=1$ and $\ba(0)=0$, then $\pi$  and $\delta_0$ are the two unique invariant probability measures;
\item if $\int_\X\bQ/\ba=\infty$, then $\delta_0$ is the unique invariant probability measure.
\end{itemize}

\

The following theorem gives estimates about the stability of these invariant measures.
Most of them will be useful for investigating Equation~\eqref{eq:lin-noncons}.

\begin{theo}
\label{th:main-lin-cons}
Under Assumptions~\eqref{as:anonnul} and~\eqref{as:bQ/ba} on $\ba$ and $\bQ$, we have the following results:
\begin{enumerate}[leftmargin=8mm,itemsep=4mm]
\item \textbf{Fast convergence.} \label{th:main-lin-cons-fast}
Assume that $\int_\X\bQ/\ba=1$ and $\ba(0)=\inf(\ba)>0$.
Then $\pi=\bQ/\ba$ is the unique invariant probability measure of $(P_t)_{t\geq0}$ and we have
\vspace{2mm}
\begin{enumerate}[leftmargin=6mm]
\item\label{th:main-lin-cons-fast-TV} For all $\mu\in\P(\X)$ and all $t\geq0$,
\[\left\| \mu P_t -\pi\right\|_{\TV} \leq \e^{-\inf(\ba) t}\left\| \mu-\pi\right\|_{\TV}.\]
\item\label{th:main-lin-cons-fast-Lr_1<r<2} For all $p\in[1,2]$, all $\mu\in \P(\X)\cap L^p(\pi^{1-p})$, and all $t\geq 0$,
 \[ \left\| \mu P_t -\pi \right\|_{L^p(\pi^{1-p})} \leq \e^{-\inf(\ba) t}\left\| \mu -\pi \right\|_{L^p(\pi^{1-p})}.\]
\item\label{th:main-lin-cons-fast-Lr_r>2} For all $p>2$, all $\mu\in \P(\X)\cap L^p(\pi^{1-p})$, and all $t\geq 0$,
 \[ \left\| \mu P_t -\pi \right\|_{L^p(\pi^{1-p})} \leq 2\,\e^{-\inf(\ba) t}\left\| \mu -\pi \right\|_{L^p(\pi^{1-p})}.\]
\item\label{th:main-lin-cons-fast-Linf} For all $\mu\in\P(\X)\cap L^1$ such that $\mu/\pi\in L^\infty$ and all $t\geq 0$,
 \[ \left\| \mu P_t /\pi -\1 \right\|_{L^\infty} \leq 2\,\e^{-\inf(\ba) t}\left\| \mu/\pi-\1 \right\|_{L^\infty}.\]
\item\label{th:main-lin-cons-fast-KL} For all $\mu\in\P(\X)\cap L^1$ such that $D_{\textrm{KL}}(\mu \| \pi)< \infty$ and all $t\geq0$,
$$ 
D_{\textrm{KL}}(\mu P_t \| \pi) \leq \e^{-\inf(\ba) t} D_{\textrm{KL}}(\mu \| \pi).
$$
\end{enumerate}

\item \label{th:main-lin-cons-slow} \textbf{Slow convergence.} Assume that $\int_\X\bQ/\ba=1$ and $\ba(0)=\inf(\ba)=0$.
Then $\pi$ and $\delta_0$ are the two invariant probability measures of $(P_t)_{t\geq0}$.
Moreover
\vspace{1mm}
\begin{enumerate}[leftmargin=6mm]
\item\label{th:main-lin-cons-slow-weak} If $\ba\in L^1(\bQ)$, then for all $\mu\in\P(\X)$
\[\mu P_t \xrightarrow[t\to+\infty]{}\mu(\{0\})\,\delta_0+(1-\mu(\{0\}))\,\pi \]
in the narrow topology.
\item\label{th:main-lin-cons-slow-TV} If $1/\ba\in L^q(\bQ)$ for some $q>1$ and $\ba\in L^\infty$, then there exists $C>0$ such that for all $\mu\in\P(\X)\cap\M(\ba^{-q})$ and all $t\geq0$,
\[\left\| \mu P_t -\pi\right\|_{\TV}\leq C t^{1-q}\left\|\mu\right\|_{\M(\1+\ba^{-q})}.\]
\item \label{th:main-lin-cons-slow-inf} If $1/\ba\in L^q(\bQ)$ for some $q>1$, then there exists $C>0$ such that for any $\mu\in\P(\X)\cap L^1$ such that $\mu/\pi\in L^\infty$, any $p\in[1,\infty)$, and all $t\geq 0$,
\[\left\| \mu P_t -\pi\right\|_{L^p(\pi^{1-p})}\leq C t^{-\frac{q-1}{p}}\left\| \mu/\pi-\1 \right\|_{L^\infty}.\]
\item \label{th:main-lin-cons-slow-r} If $1/\ba\in L^q(\bQ)$ for some $q>1$ and, for $r>p$, either $p\in[1,2]$ and $\ba\in L^{r'}(\pi)$, or $p\geq2$ and $\ba\in L^p(\pi)$, then there exists $C>0$ such that for all $\mu\in L^r(\pi^{r-1})$ and all $t\geq 0$,
$$ 
 \left\| \mu P_t -\pi \right\|_{L^p(\pi^{1-p})} \leq  C t^{-(q-1)\left(1/p -1/r\right) } \left\| \mu -\pi \right\|_{L^{r}(\pi^{1-r})}.
$$
\end{enumerate}

\item \label{th:main-lin-cons-conc} \textbf{Concentration.}
If $1/\ba\notin L^1(\bQ)$, then $\delta_0$ is the unique invariant probability measure of $(P_t)_{t\geq0}$.
If furthermore $\ba\in L^1(\bQ)$, then for all $\mu\in\P(\X)$ we have the narrow convergence
\[\mu P_t\xrightarrow[t\to+\infty]{}\delta_0.\]

\end{enumerate}
\end{theo}

\medskip

Note that the constants $C$ in~\eqref{th:main-lin-cons-slow-inf} and~\eqref{th:main-lin-cons-slow-r} are explicitly calculable, see the proofs in Section~\ref{sect:cvlente}.
%It is also worth mentioning that in~\eqref{th:main-lin-cons-slow-TV}, $\left\|\mu-\pi\right\|_{\M(\1+\ba^{-q})}$ can be infinite if $1/\ba\not\in L^{q+1}(\bQ)=L^q(\pi)$.

%%%%%%%%%%%%%%%%%%%%%%%%%%%%%%%%%%%%%%%%%%%%%%%%%%%%%%%%%%%%%%%%%%%

%\begin{Rq}[Markovian techniques for total variation decay]
%If we apply the function $V=1/a$ over the generator, we find that $LV$ is negative around $a^{-1}(\{0\})$. The semigroup then satisfies an assumption of the type \cite[CD2]{MTIII} and then the theorems of \cite[Section 5]{MTIII} should provide convergence in total variation norm. Similarly results of the previous section can also be seen as an application of \cite[CD3]{MTIII} criterion. However, these theorems suppose supplementary assumptions on the space $\X$ that we do not need here.
%\end{Rq}

\medskip

\subsection{Geometric convergence}

We treat here the case of a function $\ba$ such that
\begin{equation}\label{as:fastconvcons}
\inf_\X\ba=\ba(0)>0\qquad\text{and}\qquad\int_\X\frac{\bQ}{\ba}=1.
\end{equation}
We recall that in this case $\pi=\bQ/\ba$ is the unique invariant probability measure of $(P_t)_{t\geq0}$.
The first result of Theorem~\ref{th:main-lin-cons}-\eqref{th:main-lin-cons-fast} is proved through a coupling approach.
We refer to~\cite{C04,FP20,FP21,M15} for details on coupling techniques.
Note however that the proof below does not use these references.

\begin{proof}[Proof of Theorem~\ref{th:main-lin-cons}-(\ref{th:main-lin-cons-fast-TV})]
Consider $(\mathbb{P}_t)$ the semigroup on $\mathcal L^\infty(\X\times\X)$ generated by
\begin{align*}
\mathbb{L}f&(x,y) 
= \min(\ba(x),\ba(y)) \left( \int f(z,z) \bQ(dz) - f(x,y)\right)\\
&+ (\ba(x)-\ba(y))_+ \left( \int f(z,y) \bQ(dz) - f(x,y)\right)+ (\ba(y)-\ba(x))_+ \left( \int f(x,z) \bQ(dz) - f(x,y)\right).
\end{align*}
This means that $(\mathbb{P}_t)$ is the solution to
$$
\partial_t \mathbb{P}_t= \mathbb{P}_t \mathbb{L}= \mathbb{L} \mathbb{P}_t.
$$
This semigroup is a coupling of $(P_t)$ because if $f(x,y)=g(x)$ (resp. $g(y)$) then $\mathbb{P}_tf(x,y)= P_tg(x)$ (resp. $P_t g(y)$) for any function $g:\X \to \mathbb{R}$. Consequently,
\begin{align*}
\left\| \mu P_t -\pi\right\|_{\TV}
&= \left\| \mu P_t -\pi P_t\right\|_{\TV}
= \sup_{\|g\|_\infty \leq 1} \left| \langle \mu P_t,  g \rangle -\langle \pi P_t,  g \rangle\right| = \sup_{\|g\|_\infty \leq 1} \left| \langle \eta \mathbb{P}_t, F_g\rangle \right| \leq 2 \langle \eta \mathbb{P}_t,  \mathbf{1}_{x\neq y}\rangle,
\end{align*}
where $F_g(x,y)=g(x)-g(y)$, $\eta$ is any coupling measure of $\mu$ and $\pi$ ({\it i.e.} a probability measure on $\X\times\X$ with marginals $\mu$ and $\pi$), and $\mathbf{1}_{x\neq y}$ is the function on $\X^2$ which is $0$ on the diagonal and~$1$ outside.
Since  $\mathbb{L} \mathbf{1}_{x\neq y} \leq - \inf(\ba) \mathbf{1}_{x\neq y}$, Grönwall's lemma entails $ \mathbb{P}_t \mathbf{1}_{x\neq y} \leq \e^{-\inf(\ba) t} \mathbf{1}_{x\neq y}$ and then
\[
\left\| \mu P_t -\pi\right\|_{\TV} \leq \langle \eta, \mathbb{P}_t  \mathbf{1}_{x\neq y}\rangle\leq 2\,\e^{-\inf(\ba) t} \eta(\{x\neq y\}).
\]
Recalling that $2 (\mu - \pi)_+(\X)= 2 (\pi - \mu)_+(\X) =\left\| \mu-\pi\right\|_{\TV}$,
we choose the coupling measure
$$
\eta(dx,dy)=\frac{2 (\mu - \pi)_+(dx)  (\pi - \mu)_+ (dy)}{\left\| \mu-\pi\right\|_{\TV}} + \delta_x(dy) (\mu \wedge \pi)(dx),
$$
where $\mu \wedge\pi = \mu - (\mu-\pi)_+ = \pi - (\pi - \mu)_+$, and get the result as $\eta(\{x\neq y\}) = \left\| \mu-\pi\right\|_{\TV}/2$.

\end{proof}

The convergences in stronger norms of Theorem~\ref{th:main-lin-cons}-\eqref{th:main-lin-cons-fast} are proved through functional inequalities, of Poincaré or logarithmic Sobolev type, that are known to be a powerful tool for deriving the exponential decay of $\Phi$-entropies~\cite{toulouse,B94,BGL,CGR10,Ch04}.
We have already seen that the $\Phi$-entropies associated to the invariant measure $\pi$ decrease along the solutions  of Equation~\eqref{eq:dual-cons}.
This can also be obtained by differentiating the entropy along the trajectories and defining the dissipation of entropy as the opposite of this derivative
\begin{equation}
\label{eq:entropie-dissipation}
\frac{d}{dt}\Ent^{\Phi}_\pi(P_tf)=\int\Phi'(P_tf)LP_tf\,d\pi=-D^\Phi_\pi(P_tf).
\end{equation}
The convexity of $\Phi$ and some calculations then ensure the non-negativity of the dissipation
\begin{align*}
D^\Phi_\pi(f)&=-\int\Phi'(f)Lf\,d\pi=\int \big[L(\Phi(f))-\Phi'(f)Lf\big]d\pi\\
&=\int \big[\langle\bQ,\Phi(f)\rangle-\Phi(f)-\Phi'(f)(\langle\bQ,f\rangle-f)\big]d\bQ\\
&=\iint\big[\Phi(f(x))-\Phi(f(y))-\Phi'(f(y))(f(x)-f(y))\big]\bQ(x)\bQ(y)dxdy\geq0.
\end{align*}
When $\Phi'$ is concave we have additionally, using Jensen's inequality,
\begin{align*}
D^\Phi_\pi(f)&=-\int\Phi'(f)Lf\,d\pi=\int \Phi'(f)(f-\bQ(f))\,d\bQ\\
&= \langle \bQ,f\,\Phi'(f)\rangle - \langle \bQ,f\rangle \langle \bQ,\Phi'(f)\rangle\\
&\geq \langle \bQ,f\,\Phi'(f)\rangle - \langle \bQ,f\rangle\Phi'(\langle \bQ,f\rangle) =  \Ent^{x\mapsto x \Phi'(x)}_\bQ(f).
 \end{align*}
In the case $\Phi(x)=|x|^p$ with $p\geq1$ we have $x\Phi'(x)=p\Phi(x)$ and consequently, for any $p\in[1,2]$,
\begin{equation}
\label{eq:dissipation-p}
D^\Phi_\pi(f)\geq p\Ent^\Phi_\bQ(f).
\end{equation}
%Remark: in the specific case $p=2$ we even have
%\[D^\Phi_\pi(f)=\int\int\big(f(x)-f(y)\big)^2Q(dx)Q(dy)=2\int\big(f(x)-Q(f)\big)^2Q(dx)=2\big[Q(f^2)-(Q(f))^2\big]=2\Ent^%\Phi_Q(f).\]
In the specific case $p=2$ it is even an equality. For $\Phi(x)=x\log x$  we have $x\Phi'(x)=\Phi(x)+x$ and so for all $f>0$
\begin{equation}
\label{eq:dissipation-log}
D^\Phi_\pi(f)\geq \Ent^\Phi_\bQ(f)+\Ent^{\mathrm{Id}}_\bQ(f) = \Ent^\Phi_\bQ(f).
\end{equation}
When $\ba$ is constant, $\Ent^\Phi_\bQ= \ba \Ent^\Phi_\pi$ and Inequality~\eqref{eq:dissipation-log} ensures exponential decay of the entropy. This can be generalized for non-constant $\ba$ as shown in the following result.

\begin{prop}[$\Phi$-entropy decay]
\label{prop:logsob}
Under Assumption~\eqref{as:fastconvcons},
if $\Phi(x)= |x|^p$ with $p\in [1,2]$ or $\Phi(x)=x\log x$, then for all $f\in L^1(\pi)$ such that $\Phi(f)\in L^1(\pi)$ we have
\begin{equation}
\label{eq:logsob}
\Ent^{\Phi}_\pi(f)  \leq  \frac{1}{p\inf\ba} D^\Phi_\pi(f),
\end{equation}
where $p=1$ when $\Phi(x)=x\log x$. Consequently, for all $t\geq 0$,
\begin{equation}\label{eq:entropydecay}
\Ent^{\Phi}_\pi(P_t f) \leq \e^{-p(\inf\ba) t} \Ent^{\Phi}_\pi( f).
\end{equation}
\end{prop}

Before proving this result, let us recall that Inequality~\eqref{eq:logsob} is a powerful inequality that yields further properties than exponential convergence. Let us cite for instance concentration of measures \cite{Le99} or hypercontractivity \cite[Théorème 2.8.2]{toulouse}:  for every $r>1$
$$
\left\| P_t f \right\|_{L^{r(t)}}  \leq \left\| f \right\|_{L^{r}} 
$$
where $r(t)=1 +(r-1)e^{4t \inf(\ba)}$. Even if we prove~\eqref{eq:logsob} by a simple argument, a classical way for proving such inequality is the Bakry-Emery $\Gamma_2$ criterion \cite{B94} which implies local Poincaré type inequality but seems not applying here.

\begin{proof}[Proof of Proposition~\ref{prop:logsob}]
The exponential entropy decay~\eqref{eq:entropydecay} readily follows from the combination of~\eqref{eq:entropie-dissipation} and~\eqref{eq:logsob}.
For proving~\eqref{eq:logsob}, we use a similar argument to the classical perturbation result of Holley and Stroock for logarithmic Sobolev inequalities, see~\cite{HS86} or~\cite[Proposition 3.2]{Ch04}.
Let $f$ such that $\Ent_\pi^\Phi(f)<\infty$.
The functions $\Phi$ under interest are all differentiable outside of zero, so we can define for all $z\neq0$
\begin{align*}
F(z)&=\int \big(\Phi(f) - \Phi(z)- \Phi'(z)(f -z)\big) \,d\pi\\
&=\langle\pi,\Phi(f)\rangle-\Phi(z)-\Phi'(z)\big(\langle\pi,f\rangle-z\big)\\
&=\Ent_\pi^\Phi(f)+\Phi(\langle\pi,f\rangle)-\Phi(z)-\Phi'(z)\big(\langle\pi,f\rangle-z\big).
\end{align*}
The convexity of $\Phi$ thus guarantees that $\Ent_\pi^\Phi(f)=\inf_{z\neq0}F(z)$.
Consequently,
\begin{align*}
\text{Ent}^{\Phi}_\pi(f) &=\inf_{z\neq0} \int \big(\Phi(f) - \Phi(z)- \Phi'(z)(f -z)\big)\, d\pi
\\
&\leq \frac{1}{\inf\ba} \inf_{z\neq0}\int \big(\Phi(f) - \Phi(z)- \Phi'(z)(f -z)\big)\, d\bQ= \frac{1}{\inf \ba} \text{Ent}^\Phi_\bQ(f)
\end{align*}
which, combined with~\eqref{eq:dissipation-p} or~\eqref{eq:dissipation-log}, yields~\eqref{eq:logsob}.
\end{proof}

\begin{proof}[Proof of Theorem~\ref{th:main-lin-cons}-(\ref{th:main-lin-cons-fast-Lr_1<r<2})-(\ref{th:main-lin-cons-fast-Lr_r>2})-(\ref{th:main-lin-cons-fast-Linf})-(\ref{th:main-lin-cons-fast-KL})]

For $\Phi(x)=|x|^p$, $1\leq p\leq2$, the entropy decay \eqref{eq:entropydecay} reads for $f$ such that $\langle\pi,f\rangle=0$
\begin{equation}\label{eq:Lpdecay}
\left\|P_tf\right\|_{L^p(\pi)}\leq \e^{- \inf(\ba)t}\left\|f\right\|_{L^p(\pi)}.
\end{equation}
Choosing $f=\frac\mu\pi-\langle\mu,\1\rangle$ and using the duality relation $\mu P_t-\pi=\big(P_t\frac{\mu}{\pi}-\1\big)\pi$ yields~\eqref{th:main-lin-cons-fast-Lr_1<r<2}.
For $p>2$, we argue by the duality representation of the $L^p$ norms.
More precisely, due to Hölder's inequality, we have for any $\mu\in L^p(\pi^{1-p})$
\[\|\mu\|_{L^p(\pi^{1-p})}=\|\mu/\pi\|_{L^p(\pi)}=\sup_{\|f\|_{L^{p'}(\pi)}}\int_\X \frac\mu\pi f\,d\pi=\sup_{\|f\|_{L^{p'}(\pi)}\leq1}\langle\mu,f\rangle,\]
where we recall that $\frac1p+\frac1{p'}=1$.
Using~\eqref{eq:Lpdecay} for $p'\in(1,2)$, we get for $\mu\in \P(\X)\cap L^p(\pi^{1-p})$
\begin{align*}
\left\|\mu P_t-\pi\right\|_{L^p(\pi^{1-p})}&=\sup_{\|f\|_{L^{p'}(\pi)}\leq1}\langle\mu P_t-\pi,f\rangle=\sup_{\|f\|_{L^{p'}(\pi)}\leq1}\langle\mu-\pi,P_tf\rangle\\
&=\sup_{\|f\|_{L^{p'}(\pi)}\leq1}\big\langle\mu-\pi,P_tf-\langle\pi,f\rangle\big\rangle\\
&\leq\left\|\mu-\pi\right\|_{L^p(\pi^{1-r})}\sup_{\|f\|_{L^{p'}(\pi)}\leq1}\left\|P_tf-\langle\pi,f\rangle\right\|_{L^{p'}(\pi)}\\
&\leq\left\|\mu-\pi\right\|_{L^p(\pi^{1-p})} \,\e^{-\inf(\ba)t}\sup_{\|f\|_{L^{p'}(\pi)}\leq1}\left\|f-\langle\pi,f\rangle\right\|_{L^{p'}(\pi)}\\
&\leq2\,\e^{-\inf(\ba)t}\left\|\mu-\pi\right\|_{L^p(\pi^{1-p})}.
\end{align*}
The same method yields~\eqref{th:main-lin-cons-fast-Linf} since $\left\|\mu/\pi\right\|_{L^\infty}=\sup_{\|f\|_{L^1(\pi)\leq1}}\langle\mu,f\rangle$.
\iffalse
For the convergence in $L_{\exp(\pi)}$ we also use the same method.
For $\Phi(x)=x\log x-x$ and $f>0$ such that $\langle\pi,f\rangle=\e$, Theorem~\ref{th:logsob} ensures that
\[\int_\X\Phi(P_tf)\,d\pi\leq\e^{-\inf(\ba)t}\int_\X\Phi(f)\,d\pi,\]
and as a consequence, since $\Psi=\exp$ is the Legendre transform of $\Phi$ ($-\Phi$ en fait),
\begin{align*}
\left\|\mu P_t-\pi\right\|_{L_{\exp}(\pi)}&=\sup_{\langle\pi,\Phi(f)\rangle\leq1}\langle\mu P_t-\pi,f\rangle=\sup_{\langle\pi,\Phi(f)\rangle\leq1}\langle\mu-\pi,P_tf\rangle\\
&=\sup_{\langle\pi,\Phi(f)\rangle\leq1}\bigg(\frac{\langle\pi,f\rangle}{\e}\big\langle\mu-\pi,\frac{\e}{\langle\pi,f\rangle}P_tf\big\rangle\biggr)\\
&\leq\left\|\mu-\pi\right\|_{L_{\exp}(\pi)}\sup_{\langle\pi,\Phi(f)\rangle\leq1}\bigg(\frac{\langle\pi,f\rangle}{\e}\big\langle\mu-\pi,\frac{\e}{\langle\pi,f\rangle}P_tf\big\rangle\biggr)\\
&\leq\left\|\mu-\pi\right\|_{L^r(\pi^{1-r})}\sup_{\|f\|_{L^{r'}(\pi)}\leq1}\left\|P_tf-\langle\pi,f\rangle\right\|_{L^{r'}(\pi)}\\
&\leq\left\|\mu-\pi\right\|_{L^r(\pi^{1-r})} \,e^{-r\inf(\ba)t}\sup_{\|f\|_{L^{r'}(\pi)}\leq1}\left\|f-\langle\pi,f\rangle\right\|_{L^{r'}(\pi)}\\
&\leq2\,e^{-r\inf(\ba)t}\left\|\mu-\pi\right\|_{L^r(\pi^{1-r})}.
\end{align*}
\fi 
Finally, the Kullback-Leibler divergence decay also comes from~\eqref{eq:entropydecay} and the duality relation $\mu P_t= \big(P_t\frac{\mu}{\pi}\big) \pi$, which yields
$$
D_{\textrm{KL}} (\mu P_t \| \pi)= \Ent^{x\mapsto x\log x}_\pi\left( P_t\frac{\mu}{\pi} \right).
$$
\end{proof}

\subsection{Algebraic convergence}
\label{sect:cvlente}

In this subsection, we assume  that
\begin{equation}
\label{eq:hypcvlent}
\inf(\ba) = 0\qquad\text{and} \qquad \int_\X \frac{\bQ}{\ba}=1.
\end{equation}
In particular, we have two invariant probability distributions  $\pi=\bQ/\ba$ and $\delta_0$.
We start by proving~\eqref{th:main-lin-cons-slow-TV}.
To do so we use a subgeometric result taken from~\cite{Butkovsky2014}, which is inspired from~\cite{DFG} and was recently revisited in~\cite{CM21}.
The result of~\eqref{th:main-lin-cons-slow-TV} is a direct consequence of the next proposition.

\begin{prop}\label{th:polynomlin}
Suppose that~\eqref{eq:hypcvlent} is verified, that $\ba\in L^\infty$, and that $1/\ba\in L^q(\bQ)$ for some $q>1$.
Then there exists $C>0$ such that for any $x\in\X$ and $t\geq0$
 $$
 \| \delta_x P_t - \pi \|_{\TV} \leq \frac{C \ba(x)^{-q}}{ (t/q +1)^q } + \frac{C}{(t/q +1)^{q-1}} \leq \frac{C(1+ \ba(x)^{-q})}{(t+1)^{q-1}}.
 $$
 \end{prop}
 
 \begin{proof}
It follows from~\cite[Theorem~2.4]{Butkovsky2014}, see also~\cite[Theorem~4.1]{Ha10}, if we can prove that the two following conditions are verified:
\begin{enumerate}
\item $\L^*V\leq K-\varphi(V)$ for some $K>0$ with $V=\ba^{-q}$ and $\varphi:x \mapsto x^{\frac{q-1}{q}}$,
\item for every $C>0$, there exists $\alpha>0$ and $T>0$ such that
\[\|\delta_xP_T-\delta_yP_T\|_\TV\leq 2(1-\alpha)\]
for all $(x,y)\in\X^2$ such that $V(x)+V(y)\leq C$.
\end{enumerate}
%Note that the precompactness assumption of~\cite[Theorem 4.1]{Ha10} is only needed to prove the existence of the invariant measure, which is known here.
The first condition is clearly verified since for $V=\ba^{-q}$ we have
\[\L^*V(x)=\|1/\ba\|^q_{L^q(\bQ)}\ba(x)-\ba^{1-q}(x)\]
and $\ba$ is supposed to be bounded.
For the second condition, we first remark that for every $C>0$, the bound $V(x)+V(y)\leq C$ implies the existence of $c>0$ such that $\ba(x)\geq c$ and $\ba(y)\geq c$.
Then, for every $(x,y)$ such that $V(x)+V(y)\leq C$ and any $z\in\{x,y\}$, we have from Duhamel's formula~\eqref{eq:Duhamel-cons} that for all $f\geq0$ and all $t\geq0$
\[P_tf(z) \geq c\int_0^t \e^{-Cs} \langle \bQ, P_{t-s}f\rangle ds=\alpha\langle\nu,f\rangle\]
where $\nu(dx)=C(1-\e^{-Ct})^{-1}\int_0^t\e^{-Cs}(\bQ P_{t-s})(dx)ds$ is a probability measure and $\alpha=\frac{c(1-\e^{-Ct})}{C}>0$.
Hence, the measures $\delta_xP_t-\alpha\nu$ and $\delta_yP_t-\alpha\nu$ are positive and Doeblin's argument yields
\[\| \delta_x P_t - \delta_y P_t\|_{\TV}\leq \| \delta_x P_t - \alpha\nu \|_\TV + \| \delta_y P_t - \alpha \nu\|_\TV=\langle\delta_x P_t - \alpha \nu,\1\rangle+\langle\delta_y P_t - \alpha\nu,\1\rangle = 2(1-\alpha).\]
We can then apply~\cite[Theorem~2.4]{Butkovsky2014} which yields the result since
$$
H_\varphi(u) = \int_1^u\frac{ds}{\varphi(s)} = q(u^{1/q}-1) \qquad\text{and}\qquad H_\varphi^{-1}(t) =(t/q +1 )^q.
$$
 \end{proof}
 
As for the fast convergence, we can give stronger convergence results depending on the tails of the initial condition. This sub-geometric bounds in entropy seems more original than the previous bound in total variation distance. Let us nevertheless cite \cite{CGG07,RW01} which use others weakening of entropic inequalities to obtain algebraic convergence of some diffusion processes.

\begin{prop}
\label{prop:polynomlin-pierre}
Suppose that~\eqref{eq:hypcvlent} is verified and that $1/\ba\in L^q(\bQ)$ for some $q>1$.
Then for all $p\in[1,2]$, all $f\in L^\infty$, and all $t\geq 0$, we have
\begin{equation}\label{eq:algebraic-decay-inf}
\Ent_\pi^{|\cdot|^p}(P_tf)\leq 2p\big(p(q-1)\big)^{q-1}\Big\|\frac1\ba\Big\|_{L^q(\bQ)}^{q}\big\|f\big\|_{L^\infty}^p\,t^{-(q-1)}.
\end{equation}
If additionally $\ba\in L^{r'}(\pi)$ for some $r>p$, then for all $f\in L^r(\pi)$ and all $t\geq 0$ we have
\begin{equation}\label{eq:algebraic-decay-r}
\Ent_\pi^{|\cdot|^p}(P_tf)\leq C_r\Big\|\frac1\ba\Big\|_{L^q(\bQ)}^{q(1-\frac{p}{r})}\max\Big(1,\|\ba\|_{L^{r'}(\pi)}^p\Big)\big\|f\big\|_{L^r(\pi)}^p\,t^{-(q-1)(1-\frac pr)}
\end{equation}
where
\[C_r=2p\Big(p(q-1)\Big(1-\frac pr\Big)\Big)^{(q-1)(1-\frac pr)}.\]
\end{prop}

\begin{proof}
Consider $\Phi(x)=|x|^p$ with $p\in[1,2]$, so that
\[D^\Phi_\pi(f)\geq p\Ent^\Phi_\bQ(f),\]
and start from
\begin{align*}
\Ent^\Phi_\pi(f)&=\inf_{z\in\R}\int\big[\Phi(f)-\Phi(z)-\Phi'(z)(f-z)\big]\,d\pi\\
&\leq\int\big[\Phi(f)-\Phi(\langle\bQ,f\rangle)-\Phi'(\langle\bQ,f\rangle)(f-\langle\bQ,f\rangle)\big]\,d\pi.
\end{align*}
Let $q>1$ such that $1/\ba\in L^q(\bQ)$ and denote
\[\Xi(f)=\Phi(f)-\Phi(\langle\bQ,f\rangle)-\Phi'(\langle\bQ,f\rangle)(f-\langle\bQ,f\rangle).\]

\medskip

\paragraph{\it Proof of~\eqref{eq:algebraic-decay-inf}}
We start with the proof of~\eqref{eq:algebraic-decay-inf}, which is simpler than for~\eqref{eq:algebraic-decay-r}.
For $f\in L^\infty$ we get from Hölder's inequality that
\begin{align*}
\Ent^\Phi_\pi(f)\leq\int\Xi(f)\,d\pi&\leq \bigg(\int\big(\Xi(f)\big)^{q'}\,d\bQ\bigg)^{\frac1{q'}}\Big\|\frac1\ba\Big\|_{L^q(\bQ)}\\
&\leq \big\|\Xi(f)\big\|_{L^\infty}^{\frac1q}\bigg(\int\Xi(f)\,d\bQ\bigg)^{\frac1{q'}}\Big\|\frac1\ba\Big\|_{L^q(\bQ)}
=\Big\|\frac1\ba\Big\|_{L^q(\bQ)}\big\|\Xi(f)\big\|_{L^\infty}^{\frac1q} \big(\Ent^\Phi_\bQ(f)\big)^{\frac{1}{q'}}.
\end{align*}
Now we control $\Xi(f)$ by
\begin{align*}
\Xi(f)&=\Phi(f)-\Phi(\langle\bQ,f\rangle)-\Phi'(\langle\bQ,f\rangle)(f-\langle\bQ,f\rangle)\\
&=|f|^p-|\langle\bQ,f\rangle|^p-p\sign(\langle\bQ,f\rangle)|\langle\bQ,f\rangle|^{p-1}(f-\langle\bQ,f\rangle)\\
&\leq |f|^p+(p-1)|\langle\bQ,f\rangle|^p+p|\langle\bQ,f\rangle|^{p-1}|f|
\end{align*}
to get that
\[\big\|\Xi(f)\big\|_{L^\infty}\leq2p\|f\|_{L^\infty}^p,\]
and finally, for all $t\geq0$,
\[\Ent^\Phi_\pi(P_tf)\leq (2p)^{\frac1q}\,\Big\|\frac1\ba\Big\|_{L^q(\bQ)}\|f\|_{L^\infty}^{\frac pq}\big(\Ent^\Phi_\bQ(P_tf)\big)^{\frac{1}{q'}}.\]
This yields the differential inequality
\[\frac{d}{dt}\Ent_\pi^\Phi(P_tf)=-D_\pi^\Phi(P_tf)\leq-\frac1p\Ent^\Phi_\bQ(P_tf)\leq -\frac1p\bigg((2p)^{\frac1q}\Big\|\frac1\ba\Big\|_{L^q(\bQ)}\|f\|_{L^\infty}^{\frac pq}\bigg)^{-q'}\big(\Ent_\pi^\Phi(P_tf)\big)^{q'},\]
which ensures that
\[\Ent_\pi^\Phi(P_tf)\leq\bigg(\big(\Ent_\pi^\Phi(f)\big)^{1-q'}+(q'-1)C\,t\bigg)^{\frac{1}{1-q'}},\quad\text{where}\ C=\frac 1p\bigg((2p)^{\frac{1}{q}}\Big\|\frac1\ba\Big\|_{L^q(\bQ)}\|f\|_{L^\infty}^{\frac{p}{q}}\bigg)^{-q'}\]
and the speed of convergence is then given by $-\frac{1}{1-q'}=(q-1)$.
We finally deduce that
\[\Ent_\pi^\Phi(P_tf)\leq 2p\big(p(q-1)\big)^{q-1}\Big\|\frac1\ba\Big\|_{L^q(\bQ)}^{q}\big\|f\big\|_{L^\infty}^p\,t^{-(q-1)}.\]

\medskip

\paragraph{\it Proof of~\eqref{eq:algebraic-decay-r}}
Now we turn to the proof of~\eqref{eq:algebraic-decay-r}, which follows the same method as for~\eqref{eq:algebraic-decay-inf} but is a bit more technical.
For any $\alpha\in(0,1)$, splitting $\ba=\ba^\alpha \ba^{1-\alpha}$ and using the Hölder inequality
\[\int\Xi\,\pi=\int\Xi\,\ba^{\alpha-1}\ba^{-\alpha} \bQ\leq\|\Xi\,\ba^{\alpha-1}\|_{L^{q/(q-\alpha)}(\bQ)}\|\ba^{-\alpha}\|_{L^{q/\alpha}(\bQ)}\]
we have
\[\Ent^\Phi_\pi(f)\leq\int\Xi(f)\,d\pi\leq \bigg(\int\big(\Xi(f)\big)^{\frac{q}{q-\alpha}}\ba^{-\frac{(1-\alpha)q}{q-\alpha}}d\bQ\bigg)^{1-\frac\alpha q}\Big\|\frac1\ba\Big\|_{L^q(\bQ)}^\alpha.\]
Setting $s=\frac{q-\alpha}{(1-\alpha)q}>1$ and using the Hölder inequality
\[\int\Xi^{\frac{q}{q-\alpha}}\ba^{-\frac{(1-\alpha)q}{q-\alpha}}\bQ=\int\Xi^{\frac1{s'}}\Xi^{\frac{q}{q-\alpha}-\frac1{s'}}\ba^{-\frac1s}\bQ
\leq\big\|\Xi^{\frac1{s'}}\big\|_{L^{s'}(\bQ)}\big\|\Xi^{\frac{q}{q-\alpha}-\frac1{s'}}\ba^{-\frac1s}\big\|_{L^s(\bQ)},\]
we get, since $s'=\frac{q-\alpha}{(q-1)\alpha}$,
\[\Ent^\Phi_\pi(f)\leq \bigg(\int\Xi(f)\,d\bQ\bigg)^{\frac\alpha{q'}}\bigg(\int\big(\Xi(f)\big)^{1+\frac{\alpha}{1-\alpha}\frac1q}\ba^{-1}d\bQ\bigg)^{1-\alpha}\Big\|\frac1\ba\Big\|_{L^q(\bQ)}^\alpha,\]
where $q'=\frac{q}{q-1}$ is the conjugate Hölder exponent of $q$, and this also reads
\[\Ent^\Phi_\pi(f)\leq\Big\|\frac1\ba\Big\|_{L^q(\bQ)}^\alpha\big\|\Xi(f)\big\|_{L^{1+\frac{\alpha}{1-\alpha}\frac1q}(\pi)}^{1-\alpha+\frac\alpha q} \big(\Ent^\Phi_\bQ(f)\big)^{\frac{\alpha}{q'}}.\]
Now we control $\Xi(f)$ by
\begin{align*}
\Xi(f)%&=\Phi(f)-\Phi(\langle\bQ,f\rangle)-\Phi'(\langle\bQ,f\rangle)(f-\langle\bQ,f\rangle)\\
&=|f|^p-|\langle\bQ,f\rangle|^p-p\sign(\langle\bQ,f\rangle)|\langle\bQ,f\rangle|^{p-1}(f-\langle\bQ,f\rangle)\\
&\leq |f|^p+(p-1)|\langle\bQ,f\rangle|^p+p|\langle\bQ,f\rangle|^{p-1}|f|\\
&=2p\Big(\frac{1}{2p}|f|^p+\frac{p-1}{2p}|\langle\bQ,f\rangle|^p+\frac12|\langle\bQ,f\rangle|^{p-1}|f|\Big)
\end{align*}
which yields by convexity
\begin{align*}
\big(\Xi&(f)\big)^{1+\frac{\alpha}{1-\alpha}\frac1q}\\
&\leq(2p)^{1+\frac{\alpha}{1-\alpha}\frac1q}\Big(\frac{1}{2p}|f|^{p(1+\frac{\alpha}{1-\alpha}\frac1q)}+\frac{p-1}{2p}|\langle\bQ,f\rangle|^{p(1+\frac{\alpha}{1-\alpha}\frac1q)}+\frac12|\langle\bQ,f\rangle|^{(p-1)(1+\frac{\alpha}{1-\alpha}\frac1q)}|f|^{1+\frac{\alpha}{1-\alpha}\frac1q}\Big).
\end{align*}
We thus get that $\big\|\Xi(f)\big\|_{L^{1+\frac{\alpha}{1-\alpha}\frac1q}(\pi)}^{1-\alpha+\frac\alpha q}$ is bounded by
\[(2p)^{1-\alpha+\frac\alpha q}
\Big(\frac{1}{2p}\|f\|_{L^{p(1+\frac{\alpha}{1-\alpha}\frac1q)}(\pi)}^{p(1-\alpha+\frac\alpha q)}+\frac{p-1}{2p}|\langle\bQ,f\rangle|^{p(1-\alpha+\frac\alpha q)}+\frac12|\langle\bQ,f\rangle|^{(p-1)(1-\alpha+\frac\alpha q)}\|f\|_{L^{1+\frac{\alpha}{1-\alpha}\frac1q}(\pi)}^{1-\alpha+\frac\alpha q}\Big).\]
Then we estimate $\langle\bQ,f\rangle$ by
\[|\langle\bQ,f\rangle|\leq\int|f|\ba\,d\pi\leq\|f\|_{L^{p(1+\frac{\alpha}{1-\alpha}\frac1q)}(\pi)}\|\ba\|_{L^{[p(1+\frac{\alpha}{1-\alpha}\frac1q)]'}(\pi)}\]
to obtain, using that $\|f\|_{L^{1+\frac{\alpha}{1-\alpha}\frac1q}(\pi)}\leq\|f\|_{L^{p(1+\frac{\alpha}{1-\alpha}\frac1q)}(\pi)}$ since $p\geq1$ and $\pi$ is  probability measure,
\[\big\|\Xi(f)\big\|_{L^{1+\frac{\alpha}{1-\alpha}\frac1q}(\pi)}^{1-\alpha+\frac\alpha q}\leq C\|f\|_{L^{p(1+\frac{\alpha}{1-\alpha}\frac1q)}(\pi)}^{p(1-\alpha+\frac\alpha q)}\]
with
\[C=(2p)^{1-\alpha+\frac\alpha q}\max\Big(1,\|\ba\|_{L^{[p(1+\frac{\alpha}{1-\alpha}\frac1q)]'}(\pi)}^{p(1-\alpha+\frac\alpha q)}\Big).\]
Since $t\mapsto\|P_tf\|_{L^{p(1+\frac{\alpha}{1-\alpha}\frac1q)}(\pi)}^{p(1-\alpha+\frac\alpha q)}$ is nonincreasing by entropy property, we finally get for all $t\geq0$
\[\Ent^\Phi_\pi(P_tf)\leq C\,\Big\|\frac1a\Big\|_{L^q(\bQ)}^\alpha\|f\|_{L^{p(1+\frac{\alpha}{1-\alpha}\frac1q)}(\pi)}^{p(1-\alpha+\frac\alpha q)} \big(\Ent^\Phi_\bQ(P_tf)\big)^{\frac{\alpha}{q'}}.\]
This yields the differential inequality
\[\frac{d}{dt}\Ent_\pi^\Phi(P_tf)=-D_\pi^\Phi(P_tf)\leq-\frac1p\Ent^\Phi_\bQ(P_tf)\leq -C'\big(\Ent_\pi^\Phi(P_tf)\big)^{\frac{q'}{\alpha}},\]
where
\[C'=\frac1p\bigg(C\,\Big\|\frac1\ba\Big\|_{L^q(\bQ)}^\alpha\|f\|_{L^{p(1+\frac{\alpha}{1-\alpha}\frac1q)}(\pi)}^{p(1-\alpha+\frac\alpha q)}\bigg)^{-\frac{q'}{\alpha}},\]
which ensures that
\[\Ent_\pi^\Phi(P_tf)\leq\bigg(\big(\Ent_\pi^\Phi(f)\big)^{1-\frac{q'}{\alpha}}+\Big(\frac{q'}{\alpha}-1\Big)C'\,t\bigg)^{\frac{1}{1-\frac{q'}{\alpha}}}.\]
Choosing $\alpha$ such that $r=p(1+\frac{\alpha}{1-\alpha}\frac1q)>p$ we have
\[C'=\frac 1p\bigg(C\,\Big\|\frac1\ba\Big\|_{L^q(\bQ)}^\alpha\|f\|_{L^r(\pi)}^{\frac{pr}{p+q(r-p)}}\bigg)^{-\frac{q+\frac{p}{r-p}}{q-1}}
\quad\text{with}\quad
C=(2p)^{\frac{r}{p+q(r-p)}}\max\Big(1,\|\ba\|_{L^{r'}(\pi)}^{\frac{pr}{p+q(r-p)}}\Big),\]
and the speed of convergence is given by
\[-\frac{1}{1-\frac{q'}{\alpha}}=\frac{(q-1)(r-p)}{r}=(q-1)\Big(1-\frac pr\Big).\]
We finally deduce that
\[\Ent_\pi^\Phi(P_tf)\leq C''\Big\|\frac1\ba\Big\|_{L^q(\bQ)}^{q(1-\frac{p}{r})}\max\Big(1,\|\ba\|_{L^{r'}(\pi)}^p\Big)\big\|f\big\|_{L^r(\pi)}^p\,t^{-(q-1)(1-\frac pr)}\]
with
\[C''=2p\Big(p(q-1)\Big(1-\frac pr\Big)\Big)^{(q-1)(1-\frac pr)}\]
and the proof in complete.
\end{proof}

\begin{proof}[Proof of Theorem~\ref{th:main-lin-cons} - (\ref{th:main-lin-cons-slow-inf}) and (\ref{th:main-lin-cons-slow-r})]
We only give the proof of~\eqref{th:main-lin-cons-slow-r}, since \eqref{th:main-lin-cons-slow-inf} can be seen as the limit case ``$r=\infty$'' and proved in the exact same way, using~\eqref{eq:algebraic-decay-inf} instead of~\eqref{eq:algebraic-decay-r}.

\smallskip

Suppose that~\eqref{eq:hypcvlent} is verified, let $p \in [1,2]$, $q>1$, $r>p$, and assume that $\ba\in L^{r'}(\pi)$ and $1/\ba\in L^q(\bQ)$.
Then~\eqref{eq:algebraic-decay-r} in Proposition~\ref{prop:polynomlin-pierre} ensures the existence of $C>0$ such that for all $f\in L^r(\pi)$ with $\langle\pi,f\rangle=0$ and all $t\geq0$
\[\left\|P_tf\right\|_{L^p(\pi)}\leq C\max\big(1,\|\ba\|_{L^{r'}(\pi)}\big)\,\big\|f\big\|_{L^r(\pi)}\,t^{-(q-1)(\frac1p-\frac1r)}.\]
This inequality applied to $f=\frac\mu\pi-\langle\mu,\1\rangle$ yields~\eqref{th:main-lin-cons-slow-r} for $p\in[1,2]$, since $\mu P_t-\pi=\big(P_t\frac{\mu}{\pi}-\1\big)\pi$.
For the case $p>1$ in~\eqref{th:main-lin-cons-slow-r} we argue by duality, similarly as in the proof of Theorem~\ref{th:main-lin-cons}-\eqref{th:main-lin-cons-fast-Lr_r>2}
\begin{align*}
\left\|\mu P_t - \pi \right\|_{L^{r'}(\pi^{1-{r'}})}
&=\sup_{\|f\|_{L^{r}(\pi)}\leq1}\langle\mu P_t-\pi,f\rangle=\sup_{\|f\|_{L^{r}(\pi)}\leq1}\big\langle\mu-\pi,P_tf-\langle\pi,f\rangle\big\rangle\\
&\leq\left\|\mu-\pi\right\|_{L^{p'}(\pi^{1-p'})}\sup_{\|f\|_{L^{r}(\pi)}\leq1}\left\|P_tf-\langle\pi,f\rangle\right\|_{L^{p}(\pi)}\\
&\leq\left\|\mu-\pi\right\|_{L^{p'}(\pi^{1-p'})} C \max\big(1,\|\ba\|_{L^{r'}(\pi)}\big)t^{-(q-1)(1/p-1/r)}\sup_{\|f\|_{L^{r}(\pi)}\leq1}\left\|f-\langle\pi,f\rangle\right\|_{L^{r}(\pi)}\\
&\leq2\,C\max\big(1,\|\ba\|_{L^{r'}(\pi)}\big) t^{-(q-1)(1/p-1/r)}\left\|\mu-\pi\right\|_{L^{p'}(\pi^{1-p'})}.
\end{align*}
This gives the conclusion by replacing $r'$ by $p$, and consequently $p'$ by $r$, since $\frac{1}{p}-\frac{1}{r}=\frac{1}{r'}-\frac{1}{p'}$.
\end{proof}
 
\subsection{Weak convergence}

In this subsection, we consider that $\ba(0)=0$ and $\ba\in L^1(\bQ)$, or equivalently $\ba\in L^2(\bQ/\ba)$, and we prove the weak convergence results, {\it i.e.} for the narrow topology, of Theorem~\ref{th:main-lin-cons}.
To do so, we work in the space $L^2(\bQ/\ba)$ and use the quadratic entropy functional
\[H_2[f]=\|f\|_{L^2(\bQ/\ba)}^2=\int_\X \frac{f^2(x)}{\ba(x)}\bQ(dx)\]
which is, up to the addition of the mean, the variance under $\pi=\bQ/\ba$ (which is not necessarily a finite measure here).
As for the $\Phi$-entropies, we have the dissipation property
\[\frac{d}{dt}H_2[P_tf]=-D_2[P_tf],\quad\text{with}\quad D_2[f]=\iint(f(x)-f(y))^2 \bQ(dx)\bQ(dy)\geq0.\]
In particular, the subspace $L^2(\bQ/\ba)\cap L^1(\bQ/\ba)\subset L^1(\bQ/\ba)$ is invariant under the semigroup $(P_t)_{t\geq0}$ and the restriction of $(P_t)_{t\geq0}$ to this subset, endowed with the norm $\|f\|_{(L^1\cap L^2)(\bQ/\ba)}=\|f\|_{L^1(\bQ/\ba)}+\|f\|_{L^2(\bQ/\ba)}$, is a contraction semigroup, recalling that $(P_t)_{t\geq0}$ is a contraction in $L^1(\bQ/\ba)$.
%To prove rigorously the invariance of $L^2(\bQ/\ba)\cap L^1(\bQ/\ba)$, we consider a sequence $(f_n)_{n\in\N}\subset L^\infty\cap L^1(\bQ/\ba)$ (which is a dense invariant subset of $L^2(\bQ/\ba)\cap L^1(\bQ/\ba)$) that converges to $f$ in $L^2(\bQ/\ba)\cap L^1(\bQ/\ba)$.
%Due to the entropy property, the sequence $(P_tf_n)_{n\in\N}$ is a Cauchy sequence in $L^2(\bQ/\ba)\cap L^1(\bQ/\ba)$, so it converges to a limit $F\in L^2(\bQ/\ba)\cap L^1(\bQ/\ba)$.
%But since $P_t$ is a bounded operator in $L^1(\bQ/\ba)$, we have that $P_tf_n\to P_tf$ in $L^1(\bQ/\ba)$, and we deduce that $P_tf=F\in L^2(\bQ/\ba)\cap L^1(\bQ/\ba)$.
We start by proving a useful result which is a consequence of the above entropy property.
It is convenient to define the domain of $\L^*$ in $(L^1\cap L^2)(\bQ/\ba)$ by
\[\D(\L^*)=\big\{f\in(L^1\cap L^2)(\bQ/\ba),\ \L^*f\in (L^1\cap L^2)(\bQ/\ba)\big\}.\]

\begin{lem}\label{lem:entropyH2}
Assume that $\ba\in L^1(\bQ)$ and let $f\geq0$ in $\D(\L^*)$.
For any sequence $(t_n)_{n\geq0}$ of positive real numbers which is increasing and tends to $+\infty$, there exists a sub-sequence $(t'_n)_{n\geq0}$ and a continuous function $\rho:\R\to[0,\infty)$ such that for any $\varphi\in L^2(\ba\bQ)$ the convergence
\begin{equation}\label{eq:conv_fn}
\int_\X P_{t+t'_n}f(x)\varphi(x)\bQ(dx)\to\rho(t)\langle\bQ,\varphi\rangle
\end{equation}
holds locally uniformly in time.
\end{lem}

\begin{proof}
Let $f\geq0$ in $\D(\L^*)$.
By virtue of the entropy property, we have for all $t\geq0$
\[ \|P_tf\|_{(L^1\cap L^2)(\bQ/\ba)}\leq\|f\|_{(L^1\cap L^2)(\bQ/\ba)}\]
and
\[\|\partial_tP_tf\|_{(L^1\cap L^2)(\bQ/\ba)}=\|P_t\L^*f\|_{(L^1\cap L^2)(\bQ/\ba)}\leq\|\L^*f\|_{(L^1\cap L^2)(\bQ/\ba)}.\]
Consequently, due to Arzel\`{a}-Ascoli and Banach-Alaoglu theorems, we can extract, from the sequence $(f_n)_{n\geq 0}$, defined by $f_n(t,x)=P_{t+t_n}f(x)$, a sub-sequence, still denoted $(f_n)$, such that for all $\varphi\in L^2(\ba\bQ)$ the convergence
\[\int f_n(t,x)\varphi(x)\bQ(dx)\to\int f_\infty(t,x)\varphi(x)\bQ(dx),\]
holds locally uniformly in time, for some $f_\infty:\R\to (L^1\cap L^2)(\bQ/\ba)$ non-negative and weakly continuous.
Now we check that $f_\infty$ is constant in $x$.
Let $(g_n)_{n\geq 0}$  be the sequence of three variables functions defined by $g_n(t,x,y)=f_n(t,x)-f_n(t,y)$.
For all $\psi\in L^2(\bQ(dx)\otimes \bQ(dy))$ and $T>0$, we have
\begin{align*}
\left| \int_0^T\!\!\!\!\iint g_n(t,x,y)\psi(x,y)\bQ(dx)\bQ(dy)dt \right| \leq \int_0^T\!\!\!\sqrt{D_2[f_n(t,\cdot)]}dt\sqrt{\iint\psi^2(x,y)\bQ(dx)\bQ(dy)}\xrightarrow[n\to\infty]{}0,
\end{align*}
because
\[\int_0^T\!\! \sqrt{D_2[f_n(t,\cdot)]}dt = \int_{t_n}^{T+t_n}\!\!\sqrt{D_2[f(t,\cdot)]}dt \leq \sqrt{T} \sqrt{\int_{t_n}^{T+t_n}\! D_2[f(t,\cdot)]dt}\]
due to the Cauchy-Schwarz inequality, and
\[\int_0^\infty\! D_2[f(t,\cdot)]dt\leq H_2[f]<\infty.\]
Choosing now $\psi(x,y)=\xi(x)$ with $\xi\in L^2(\bQ)\cap L^2(\ba\bQ)$ such that $\langle\bQ,\xi\rangle=0$, we obtain
\[\int_0^T\!\!\int\! f_\infty(t,x)\xi(x)\bQ(dx)dt=\lim_{n\to\infty}\int_0^T\!\!\int\! f_n(t,x)\xi(x) \bQ(dx)dt=0,\]
This ensures that $f_\infty(t,x)=\rho(t)$, for every $t\geq 0$ and almost all $x\in \X$, for some non-negative continuous function $\rho$.
\end{proof}

We are now in position to prove the weak-* convergences results of Theorem~\ref{th:main-lin-cons}.

\begin{proof}[Proof of Theorem~\ref{th:main-lin-cons}-(\ref{th:main-lin-cons-slow-weak})]
We assume here that $\int\bQ/\ba=1$, $\ba(0)=0$, and $\ba\in L^1(\bQ)$.
In this case $\pi=\bQ/\ba$ is an invariant probability measure and $\langle\pi,P_tf\rangle=\langle\pi,f\rangle$ for all $f\in\D(\L^*)$.
This enforces, by taking $\varphi=1/\ba$ in~\eqref{eq:conv_fn}, that $\rho(t)=\langle\pi,f\rangle$ for all $t$, and we can infer that for any $\varphi\in L^2(\ba\bQ)$ and $f\in\D(\L^*)$,
\[\int_\X P_tf(x)\varphi(x)\bQ(dx)\xrightarrow[t\to+\infty]{}\langle\pi,f\rangle\,\langle\bQ,\varphi\rangle.\]
In particular by choosing $\varphi=\1$, which is admissible since we supposed $\ba\in L^1(\bQ)$, this ensures that $\langle\bQ,P_tf\rangle\to\langle\pi,f\rangle$ as $t\to\infty$.
From Duhamel's formula~\eqref{eq:Duhamel-cons}, we have for any continuous function $f\in\D(\L^*)$ and for all $x\in\X\setminus\{0\}$
\[P_tf(x) = f(x)\e^{-\ba(x)t} + \int_{t/2}^t \ba(x)\e^{-\ba(x)s} \langle \bQ,P_{t-s}f\rangle ds  + \int_0^{t/2} \ba(x)\e^{-\ba(x)s} \langle \bQ,P_{t-s}f\rangle ds.\]
Since $\langle\bQ,P_tf\rangle\to\langle\pi,f\rangle$ as $t\to\infty$, the first two terms tend to $0$ and the last term tends to $\langle\pi,f\rangle$.
We thus get that $P_tf(x)\to\langle\pi,f\rangle$ when $t\to\infty$ for any $x\in\X\setminus\{0\}$.
On the other hand, since $\ba(0)=0$, we have $P_tf(0)=f(0)$ for all $t\geq0$.
Finally, as $C_b(\X)\subset\D(\L^*)$ and $(P_t)_{t\geq0}$ is a Markov semigroup, we obtain by dominated convergence that for any $\mu\in\P(\X)$ and any $f\in C_b(\X)$
\[\langle\mu P_t,f\rangle\xrightarrow[t\to+\infty]{} \mu(\{0\})+(1-\mu(\{0\}))\langle\pi,f\rangle\]
which is the desired result.
\end{proof}

\begin{proof}[Proof of Theorem~\ref{th:main-lin-cons}-(\ref{th:main-lin-cons-conc})]
%We now verify that this function $\rho$ is identically zero.
%On the first hand, using~\eqref{eq:conv_fn} we have that for all $\epsilon>0$ and $t\geq0$,
%\[\int_{\X\cap\{|x|>\epsilon\}}f_n(t,x)\frac{\bQ(dx)}{\ba(x)}\xrightarrow[n\to\infty]{}\rho(t)\int_{\X\cap\{|x|>\epsilon\}}\frac{\bQ(dx)}{\ba(x)}.\]
%On the second hand, by stationarity of $\bQ/\ba$, we also have
%\[\int_{\X\cap\{|x|>\epsilon\}}f_n(t,x)\frac{\bQ(dx)}{\ba(x)}\leq\int_{\X}f_n(t,x)\frac{\bQ(dx)}{\ba(x)}=\int_{\X}f(x)\frac{\bQ(dx)}{\ba(x)}.\]
%Consequently, for all $\epsilon>0$,
%\[0\leq\rho(t)\int_{\X\cap\{|x|>\epsilon\}}\frac{\bQ(dx)}{\ba(x)}\leq\int_{\X}f(x)\frac{\bQ(dx)}{\ba(x)}<\infty.\]
%As $1/\ba\not\in L^1(\bQ)$, this enforces $\rho(t)=0$.
Here we assume that $\ba\in L^1(\bQ)$ and $\int\bQ/\ba=+\infty$.
Then $\bQ/\ba$ is not a finite measure, and a continuous function $f$ that belongs to $L^2(\bQ/\ba)$ necessarily vanishes at $0$.
Since we know from the proof of Lemma~\ref{lem:entropyH2} that for all $t$ the constant function $x\mapsto f_\infty(t,x)=\rho(t)$ belongs to $(L^1\cap L^2)(\bQ/\ba)$, we must have $\rho(t)=0$ and we deduce that for any $\varphi\in L^2(\ba\bQ)$ and $f\in\D(\L^*)$,
\[\int_\X P_tf(x)\varphi(x)\bQ(dx)\xrightarrow[t\to+\infty]{}0.\]
In particular by choosing $\varphi=\1$, which is admissible since we supposed $\ba\in L^1(\bQ)$, this ensures that $\langle\bQ,P_tf\rangle\to0$ as $t\to\infty$.
We deduce from Duhamel's formula~\eqref{eq:Duhamel-cons} and the same argument as in the proof of Theorem~\ref{th:main-lin-cons}-(\ref{th:main-lin-cons-slow-weak}) that for any continuous function $f\in\D(\L^*)$ and for all $x\in\X\setminus\{0\}$
\begin{equation}\label{eq:Duhamel-conv}
P_tf(x) = f(x)\e^{-\ba(x)t} + \int_0^t \ba(x)\e^{-\ba(x)s} \langle \bQ,P_{t-s}f\rangle ds\xrightarrow[t\to+\infty]{}0.
\end{equation}
Since a continuous function $f$ which belongs to $\D(\L^*)$ necessarily verifies $f(0)=0$, and since $\ba(0)=0$, this convergence actually holds for all $x\in\X$.

Now, let $\mu\in\P(\X)$.
Due to~\eqref{as:bQ/ba}, we can find in $\D(\L^*)$ a continuous function $f\geq0$ such that $f(x)=1$ for all $|x|\geq1$.
We deduce from~\eqref{eq:Duhamel-conv} and by dominated convergence, since $(P_t)_{t\geq0}$ is a Markov semigroup, that $\langle \mu P_t,f\rangle=\langle\mu,P_tf\rangle\to0$ and this prevents the mass of $\mu P_t$ to go to infinity.
As a consequence, we can extract from $(\mu P_t)_{t\geq0}$ a sub-sequence which converges narrowly to some $\mu_\infty\in\P(\X)$.
Since $\langle \mu P_t,f\rangle\to0$ for all bounded continuous functions $f$ with support that does not contain zero, this limit $\mu_\infty$ must be $\delta_0$ and finally the whole trajectory $(\mu P_t)_{t\geq0}$ converges to $\delta_0$ in the narrow topology.
\end{proof}

\section{The non-conservative linear equation}
\label{sec:noncons}

We consider now the non-conservative linear equation~\eqref{eq:lin-noncons} and we prove Theorem~\ref{th:main-lin-noncons}.
Denoting by $(M_t)_{t\geq0}$ the semigroup generated by $\A^*$ defined in~\eqref{eq:A*}, we have similarly as for the conservative equation~\eqref{eq:cons} that the unique solution to~\eqref{eq:lin-noncons} with initial data $u_0$ is given by $u_t=u_0M_t$, in the suitable Banach spaces.
For proving Theorem~\ref{th:main-lin-noncons}, we perform a so-called $h$-transform of $(M_t)_{t\geq0}$ and use the results of Theorem~\ref{th:main-lin-cons}.

\begin{proof}[Proof of Theorem~\ref{th:main-lin-noncons}-(\ref{th:main-lin-noncons-fast})]
Here we consider the case $\rho=\int_\X\frac{Q}{a}\in(1,+\infty]$, so that $\lambda$, defined by the relation $\int_\X\frac{Q}{\lambda+a}=1$, is strictly positive.
We have already seen in the introduction that the function
\[h(x)=\frac{\alpha^{-1}}{\lambda+a(x)}\]
verifies $\A^*h=\lambda h$ and that choosing $\alpha=\int_\X\frac{Q}{(\lambda+a)^2}$ we have $\langle\gamma,h\rangle=1$ where
\[\gamma(dx)=\frac{Q(dx)}{\lambda+a(x)}\]
verifies $\A\gamma=\lambda\gamma$.
Performing a $h$-transform of $(M_t)_{t\geq0}$ consists in defining
\[P_tf=\frac{M_{\alpha t}(fh)}{M_{\alpha t}h}=\e^{-\lambda\alpha t}\frac{M_{\alpha t}(fh)}{h}.\]
The new family $(P_t)_{t\geq0}$ is a Markov semigroup with infinitesimal generator given by
\[\alpha\,\frac{\A^*(fh)-\lambda fh}{h}=\ba\,\bigg(\int f\,d\bQ-f\bigg)=\L^*f,\]
where
\[\ba(x)=\alpha\frac{\langle Q,h\rangle}{h(x)}=\alpha\big(\lambda+a(x)\big)\qquad\text{and}\qquad \bQ(dx)=\frac{h(x)Q(dx)}{\langle Q,h\rangle}=\gamma(dx).\]
The time scaling with parameter $\alpha$ in the $h$-transform ensures that the invariant measure of $\L^*$
\[\pi(dx)=\frac{\bQ(dx)}{\ba(x)}\]
is a probability measure.
The hypotheses of Theorem~\ref{th:main-lin-cons}-\eqref{th:main-lin-cons-fast} are then satisfied, and it yields the results of Theorem~\ref{th:main-lin-noncons}-\eqref{th:main-lin-noncons-fast} by using the relation
\[\e^{-\lambda t}u_t=\e^{-\lambda t}u_0M_t=\frac{(u_0h)P_{t/\alpha}}{h}.\]
\end{proof}

\begin{proof}[Proof of Theorem~\ref{th:main-lin-noncons}-(\ref{th:main-lin-noncons-slow})]
We consider now that $\rho=\int_\X\frac{Q}{a}=1$, so that $\lambda=0$, together with the assumption that $1/a\in L^2(Q)$.
We can then still define $h$ and $\gamma$ such that $\langle\gamma,h\rangle=1$ by setting
\[h(x)=\frac{\alpha^{-1}}{a(x)}\qquad\text{and}\qquad\gamma(dx)=\frac{Q(dx)}{a(x)}\]
with $\alpha=\|1/a\|^2_{L^2(Q)}$.
Similarly as for the proof of Theorem~\ref{th:main-lin-noncons}-\eqref{th:main-lin-noncons-fast}, the estimates in Theorem~\ref{th:main-lin-noncons}-\eqref{th:main-lin-noncons-slow-TV}-\eqref{th:main-lin-noncons-slow-inf}-\eqref{th:main-lin-noncons-slow-r} are then direct consequences of Theorem~\ref{th:main-lin-cons}-\eqref{th:main-lin-cons-slow-TV}-\eqref{th:main-lin-cons-slow-inf}-\eqref{th:main-lin-cons-slow-r}.

The vague and narrow convergences in Theorem~\ref{th:main-lin-noncons}-\eqref{th:main-lin-noncons-slow-weak} readily follow from Theorem~\ref{th:main-lin-cons}-\eqref{th:main-lin-cons-slow-weak} since if $u_0\in\M(h)$ then necessarily $u_0(\{0\})=0$, and if $a\in L^\infty$ then any $f\in C_b(\X)$ satisfies $f/h\in C_b(\X)$.
For Theorem~\ref{th:main-lin-noncons}-(\ref{th:main-lin-noncons-linear-growth}'), we use the Duhamel formula
\[M_tf(x)=f(x)\e^{-a(x)t}+\int_0^t\e^{-a(x)s}\langle Q,M_{t-s}f\rangle\,ds.\]
The vague convergence in Theorem~\ref{th:main-lin-noncons}-\eqref{th:main-lin-noncons-slow-weak} ensures that $\langle Q,M_{t}f\rangle\to\langle Q,h\rangle\langle\gamma,f\rangle=\alpha^{-1}\langle\gamma,f\rangle$ as $t\to+\infty$ for all $f\in C_c(\X)$ and consequently,
since $a(0)=0$,
\[M_tf(0)=f(0)+\int_0^t\langle Q,M_{t-s}f\rangle\,ds\sim\alpha^{-1}\langle\gamma,f\rangle\, t\qquad\text{as}\ t\to+\infty.\]
When $a\in L^\infty$, we can replace $C_c(\X)$ by $C_b(\X)$ since the convergence in Theorem~\ref{th:main-lin-noncons}-\eqref{th:main-lin-noncons-slow-weak} holds narrowly, and the proof is complete.
\end{proof}

\begin{proof}[Proof of Theorem~\ref{th:main-lin-noncons}-(\ref{th:main-lin-noncons-weak-1/a})]
We consider the case $\rho=1$ with $1/a\not\in L^2(Q)$.
The function $h(x)=1/a(x)$ is an eigenfunction of $\A^*$ associated to the eigenvalue $\lambda=0$, but it cannot be normalized in such a way that $\langle\gamma,h\rangle=1$.
Nevertheless, we can perform the $h$-transform
\[P_tf=\frac{M_t(fh)}{M_t(h)}=aM_t(f/a).\]
The family $(P_t)_{t\geq0}$ is a Markov semigroup with infinitesimal generator given by
\[\L^*f=\ba\,\bigg(\int f\,d\bQ-f\bigg),\]
where
\[\ba(x)= a(x)\qquad\text{and}\qquad \bQ(dx)=\frac{Q(dx)}{a(x)}.\]
We are then in a situation where $1/\ba\not\in L^1(\bQ)$, so Theorem~\ref{th:main-lin-cons}-\eqref{th:main-lin-cons-conc} ensures that for all $\mu\in\M(\X)$ and all $f\in C_c(\X)$, or for all $f\in C_b(\X)$ if $a$ is bounded, $\langle \mu,P_t(af)\rangle \to \mu(\{0\})a(0)f(0)=0$ as $t\to+\infty$.
Besides, we have
\[\partial_tM_t(1/a)=M_t\A^*(1/a)=(\rho-1)M_t\1\leq0,\]
from which we get that $M_t(1/a)\leq1/a$ for all $t\geq0$.
Finally, we deduce that for any $u_0\in\M(1/a)$ and all $f\in C_c(\X)$, or all $f\in C_b(\X)$ if $a\in L^\infty$,
\[|\langle u_t,f\rangle|=|\langle u_0,M_tf\rangle|=|\langle M_t(1/a)u_0,P_t(af)\rangle|\leq\langle |u_0|/a,P_t(a|f|)\rangle\xrightarrow[t\to+\infty]{}0,\]
which concludes the proof of Theorem~\ref{th:main-lin-noncons}-\eqref{th:main-lin-noncons-weak-1/a}.
\end{proof}

\begin{proof}[Proof of Theorem~\ref{th:main-lin-noncons}-(\ref{th:main-lin-noncons-weak-1})]
Here we consider the case $\rho<1$.
On the one hand we have from Duhamel's formula that
\[M_t\1(x)=\e^{-a(x)t}+\int_0^t\e^{-a(x)s}\langle Q,M_{t-s}\1\rangle\,ds\leq 1+\frac{1}{a(x)}\sup_{s\in[0,t]}\langle Q,M_s\1\rangle,\]
which gives by integration against $Q$
\[\langle Q,M_t\1\rangle\leq 1+\rho\sup_{s\in[0,t]}\langle Q,M_s\1\rangle,\]
from which we infer that
\begin{equation}\label{eq:borne-inf}
\sup_{s\geq0}\langle Q,M_s\1\rangle\leq\frac{1}{1-\rho}.
\end{equation}
On the other hand, coming back to $\partial_tM_t(1/a)=(1-\rho)M_t\1$, we deduce that for all $t\geq0$
\[\int_0^t M_s\1\,ds=\frac{1}{1-\rho}\big(1/a-M_t(1/a)\big)\leq\frac{1}{(1-\rho)a},\]
which yields
\begin{equation}\label{eq:borne-int}
\int_0^{+\infty}\langle Q,M_s\1\rangle\,ds\leq\frac{\rho}{(1-\rho)},
\end{equation}
thus guaranteeing that the positive measure defined by
\[\mu(dx)=\int_0^\infty(Q M_s)(dx)\,ds\]
is finite.
Using~\eqref{eq:borne-inf} in Duhamel's formula, we get that for all $x\neq0$
\begin{align*}
M_t\1(x)&=\e^{-a(x)t}+\int_0^{t/2}\e^{-a(x)s}\langle Q,M_{t-s}\1\rangle\, ds+\int_{t/2}^t\e^{-a(x)s}\langle Q,M_{t-s}\1\rangle\, ds\\
&\leq \e^{-a(x)t}+\int_{t/2}^t\langle Q,M_{s}\1\rangle \,ds+\frac{1}{1-\rho}\int_{t/2}^t\e^{-a(x)s} \,ds,
\end{align*}
and the last three terms tend to zero as $t$ goes to infinity, by using~\eqref{eq:borne-int} for the second one.
Consequently, for all $f\in C_b(\X)$ and all $x\neq0$ we have $|M_tf(x)|\leq\|f\|_\infty M_t\1(x)\to0$ as $t\to+\infty$.
Since for $x=0$ we have
\[M_tf(0)=f(0)+\int_0^t\langle Q,M_sf\rangle \,ds\xrightarrow[t\to+\infty]{} f(0)+\langle\mu,f\rangle.\]
We can apply this convergence result to $M_s f$, for any $s\geq 0$, instead of $f$, because $\| M_s \mathbf{1} \|_\infty <+\infty$, to obtain that
$$
\left\langle \delta_0 + \mu , f \right\rangle = \left\langle \delta_0 + \mu , M_s f \right\rangle.
$$
Consequently $\delta_0 + \mu $ is an eigenvector of $\A$ and, by uniqueness, is then equal to $\gamma$ up to the multiplicative constant $(1-\rho)$.
The proof of Theorem~\ref{th:main-lin-noncons}-\eqref{th:main-lin-noncons-weak-1} is complete.
\end{proof}

\begin{Rq}[An example]
Duhamel's formula gives that $M_t f(x) \geq \e^{-t a(x)} f(x)$ for any non-negative function $f$ and $t\geq 0$. When $\X=[0,1]$, $Q(x)=1$, $a(x)=x^p$ with $p \in (0,1)$ and $f(x)=x^q$ with $q>-1$, this gives
\[\int_0^1 M_tf(x)\,dx\geq\int_0^1\e^{-tx^p}x^qdx=\frac{t^{-\frac{1+q}{p}}}{p}\int_0^t\e^{-y}y^{\frac{q+1}{p}-1}dy.\]
We proved that the left-hand side tends to $0$, since $\rho=1/(p+1)<1$, and we see that the rate of convergence in total variation distance is slower than any polynomial rate.
\end{Rq}

\section{The nonlinear conservative equation}
\label{sec:nonlin}

We turn now to the nonlinear replicator-mutator equation~\eqref{eq:hoc}.
For $v_0\in\P(\X)$, we say that $(v_t)_{t\geq0}$ is solution to Equation~\eqref{eq:hoc} if it belongs to $C([0,T],\P(\X))$, $t\mapsto av_t$ belongs to $L^1([0,T],\M(\X))$, and for all $t\geq0$
\begin{equation}\label{eq:solution}
v_t=\e^{-(a+1)t}v_0+\int_0^t\e^{-(a+1)(t-s)}\big(Q+\langle v_s,a\rangle v_s\big)\,ds.
\end{equation}
We can prove, as in~\cite{AlfGabKav} where the replicator-mutator equation with convolutive mutations is studied, that for any $v_0\in\P(\X)$ there is a unique solution to Equation~\eqref{eq:hoc} which is given by
\[v_t=\frac{u_t}{\langle u_t,\1\rangle},\]
where $u_t=v_0M_t$ is the unique solution to Equation~\eqref{eq:lin-noncons} with initial datum $v_0$.
Now we give the proof of Theorem~\ref{th:main-nonlin}.

\begin{proof}[Proof of Theorem~\ref{th:main-nonlin}-(\ref{th:main-nonlin-fast})-(\ref{th:main-nonlin-slow-TV})-(\ref{th:main-nonlin-slow-inf})-(\ref{th:main-nonlin-slow-r})]
These results are consequences of Theorem~\ref{th:main-lin-noncons}-\eqref{th:main-lin-noncons-fast}-\eqref{th:main-lin-noncons-slow-TV}-\eqref{th:main-lin-noncons-slow-inf}-\eqref{th:main-lin-noncons-slow-r} by writing
\begin{align*}
\|v_t-\gamma\|&=\left\|\frac{\e^{-\lambda t}u_t}{\langle\e^{-\lambda t}u_t,\1\rangle}-\frac{\langle v_0,h\rangle\gamma}{\langle v_0,h\rangle}\right\|\\
&=\left\|\frac{\langle\langle v_0,h\rangle\gamma-\e^{-\lambda t}u_t,\1\rangle\e^{-\lambda t}u_t+\langle\e^{-\lambda t}u_t,\1\rangle\big(\e^{-\lambda t}u_t-\langle v_0,h\rangle\gamma\big)}{\langle\e^{-\lambda t}u_t,\1\rangle\langle v_0,h\rangle}\right\|\\
&\leq\frac{\|\e^{-\lambda t}u_t\|}{\langle\e^{-\lambda t}u_t,\1\rangle}\frac{|\langle\e^{-\lambda t}u_t-\langle v_0,h\rangle\gamma,\1\rangle|}{\langle v_0,h\rangle}+\frac{\|\e^{-\lambda t}u_t-\langle v_0,h\rangle\gamma\|}{\langle v_0,h\rangle},
\end{align*}
if the convergence $\|\e^{-\lambda t}u_t-\langle v_0,h\rangle\gamma\|\to0$ yields the convergence $|\langle\e^{-\lambda t}u_t-\langle v_0,h\rangle\gamma,\1\rangle|\to0$ with the same speed.
It is true for the norm of $\M(h)$ if $a\in L^\infty$ due to the inequality
\[\|u\|_\TV\leq\|1/h\|_\infty\|u\|_{\M(h)}=\alpha\|\lambda+a\|_\infty\|u\|_{\M(h)}.\]
In $L^p(\gamma^{1-p}h)$ with $p\in[1,2)$, it is true if $a\in L^{p'-2}(Q)$ by virtue of Hölder's inequality
\[\|u\|_{L^1}\leq\|1/h\|_{L^{p'}(\gamma h)}\|u/\gamma\|_{L^p(\gamma h)}=\alpha^{\frac1p}\Big(\int_\X(\lambda+a)^{p'-2}Q\Big)^{1/p'}\|u\|_{L^p(\gamma^{1-p}h)}.\]
For $p\in[2,\infty]$ we have by Jensen's inequality
\[\int_\X(\lambda+a)^{p'-2}Q\leq\Big(\int_\X\frac{Q}{\lambda+a}\Big)^{2-p'}=1,\]
so that $\|\cdot\|_{L^1}\leq\alpha^{1/p} \|\cdot\|_{L^p(\gamma^{1-p}h)}$ without needing further conditions than {\bf(H$a$)} and {\bf(H$Q$)}.
\end{proof}

It remains to prove~\eqref{th:main-nonlin-slow-weak} and~\eqref{th:main-nonlin-conc} of~Theorem~\ref{th:main-nonlin}.
Note that when $v_0(\{0\})>0$, these results easily follow from Theorem~\ref{th:main-lin-noncons}-\eqref{th:main-lin-noncons-weak-1}.
We even get a better result than \eqref{th:main-nonlin-conc} in this case, namely a convergence without Cesàro mean.
However the situation is trickier when $v_0(\{0\})=0$ and we need the two following lemmas.

\begin{lem}\label{lem:liminf}
Assume that $\rho\leq1$. Then, for all $v_0\in\P(\X)$ and all $f\in C_b(\X)$ non-negative,
\[\liminf_{t\to+\infty}\langle v_t,f\rangle\geq\langle Q/a,f\rangle.\]
\end{lem}

\begin{proof}
Since $v_t\geq0$, we deduce from~\eqref{eq:solution} that for all $t\geq0$
\[v_t\geq Q\int_0^t\e^{-(a+1)(t-s)}ds=\frac{Q}{a+1}\big(1-\e^{-(a+1) t}\big)\geq\big(1-\e^{- t}\big)\frac{Q}{a+1}.\]
Consequently, for any $\epsilon\in(0,1)$, we have for all $t\geq t_1= -\log\epsilon$,
\[v_t\geq(1-\epsilon)\frac{Q}{a+1}.\]
Using this estimate to bound from below the quantity $\langle v_s,a\rangle$ in~\eqref{eq:hoc}, we get that for all $t\geq t_1$
\[v_t\geq \frac{Q}{a+1}\big(1-\e^{-(a+1) (t-t_1)}\big)+\alpha_1\int_{t_1}^t\e^{-(a+1)(t-s)} v(s)\,ds\]
where
\[\alpha_1=(1-\epsilon)\int_\X\frac{aQ}{a+1}<1-\epsilon.\]
Grönwall's lemma then yields that for all $t\geq t_1,$
\[v(t)\geq\frac{Q}{a+1-\alpha_1}\big(1-\e^{-(a+1-\alpha_1)(t-t_1)}\big).\]
Setting $t_2=t_1+\frac{\log\epsilon}{\alpha_1-1}$ we get that for all $t\geq t_2$
\[v_t\geq(1-\epsilon)\frac{Q}{a+1-\alpha_1}.\]
Defining the sequences $(t_n)$ and $(\alpha_n)$ by $\alpha_0=t_0=0$ and
\[\alpha_{n+1}=(1-\epsilon)\int_\X\frac{aQ}{a+1-\alpha_n}<1-\epsilon\quad\text{and}\quad t_{n+1}=t_n+\frac{\log\epsilon}{\alpha_n-1}\]
we have by induction that for all $t\geq t_n$,
\[v_t\geq(1-\epsilon)\frac{Q}{a+1-\alpha_n}\]
and consequently, for all $f\geq0$ in $C_b(\X)$, all $\epsilon\in(0,1)$, and all $n\in\N$,
\[\liminf_{t\to+\infty}\langle v_t,f\rangle\geq\langle Q/(a+1-\alpha_n),f\rangle.\]
We now study the sequence $(\alpha_n)$.
We deduce by a simple induction from
\[\alpha_{n+1}-\alpha_n=(\alpha_n-\alpha_{n-1})\,(1-\epsilon)\int_\X\frac{aQ}{(a+1-\alpha_n)(a+1-\alpha_{n-1})},\]
that $\alpha_{n+1}-\alpha_n>0$.
The sequence $(\alpha_n)_{n\geq 0}$ is increasing and bounded; it then converges to a limit $\ell_\epsilon\in(0,1-\epsilon]$ which satisfies
\[\ell_\epsilon=(1-\epsilon)\int\frac{aQ}{a+1-\ell_\epsilon}=(1-\epsilon)\bigg(1-(1-\ell_\epsilon)\int_\X\frac{Q}{a+1-\ell_\epsilon}\bigg).\]
We now study the function $\epsilon\mapsto\ell_\epsilon$.
For $0<\epsilon_1<\epsilon_2<1$ we have
\begin{align*}
\ell_{\epsilon_1}-\ell_{\epsilon_2}&=(1-\epsilon_1)(\ell_{\epsilon_1}-\ell_{\epsilon_2})\int_\X\frac{aQ}{(a+1-\ell_{\epsilon_1})(a+1-\ell_{\epsilon_2})}+(\epsilon_2-\epsilon_1)\int_\X\frac{aQ}{a+1-\ell_{\epsilon_2}}\\
&=A(\ell_{\epsilon_1}-\ell_{\epsilon_2})+B
\end{align*}
with
\[A=(1-\epsilon_1)\int_\X\frac{aQ}{(a+1-\ell_{\epsilon_1})(a+1-\ell_{\epsilon_2})}<\rho\leq1\qquad\text{and}\qquad B>0,\]
which yields $\ell_{\epsilon_1}-\ell_{\epsilon_2}\geq\frac{B}{1-A}>0$.
The function $\epsilon\mapsto\ell_\epsilon$ is then decreasing and it converges, when $\epsilon\to0$,  to a limit $\ell_0\in(0,1]$ which satisfies
\[1-\ell_0=(1-\ell_0)\int_\X\frac{Q}{a+1-\ell_0}.\]
Since $\int_\X\frac{Q}{a+1-\ell_0}<\rho\leq1$ if $\ell_0<1$, we necessarily have $\ell_0=1$, and finally for any $f\geq0$ in $C_b(\X)$
\[\liminf_{t\to+\infty}\langle v_t,f\rangle\geq\sup_{\epsilon\in(0,1)}\sup_{n\geq0}\langle Q/(a+1-\alpha_n),f\rangle=\sup_{\epsilon\in(0,1)}\langle Q/(a+1-\ell_\epsilon),f\rangle=\langle Q/a,f\rangle.\]
\end{proof}

\begin{lem}\label{lem:limsup}
For all $v_0\in\P(\X)$ we have
\[\limsup_{t \to \infty} \frac{1}{t} \int_0^t \langle v_s,a\rangle\, ds\leq 1.\]
\end{lem}

\begin{proof}
Let $t_0>0$.
From~\eqref{eq:solution} we readily see that $v_{t_0}\geq\e^{-t_0(a+1)}Q$ and so, by virtue of {\bf(H$\X$)}-{\bf(H$Q$)}-{\bf(H$a$)}, for any $\epsilon>0$ there exists a set $A_\epsilon\subset\X$ such that
\[\int_{A_\epsilon}v_{t_0}>0\qquad\text{and}\qquad\sup_{A_\epsilon}a\leq\epsilon.\]
Defining $w_t=v_t-Q/a$, Equation~\eqref{eq:hoc} also reads
\[\partial_tw_t=\partial_tv_t=-aw_t+\langle w_t,a\rangle v_t,\]
and so for all $t\geq t_0$
\[v_t=v_{t_0}+\int_{t_0}^t\big(\langle w_s,a\rangle-a \big)v_s\,ds.\]
Integrating over $A_\epsilon$ we get by Grönwall's lemma
\[\Big(\int_{A_\epsilon}v_{t_0}\Big)\,\e^{\int_{t_0}^t\langle w_s,a\rangle ds-\epsilon t}\leq \int_{A_\epsilon}v_t\leq 1\]
and consequently
\[ \frac{1}{t} \int_{t_0}^t\langle w_s,a\rangle\,ds\leq \epsilon - \frac{\log\left(\int_{A_\epsilon}v_0 \right)}{t},\]
Taking first the $\limsup$ as $t\to+\infty$ and letting then $\epsilon$ go to zero, we find that
\[\limsup_t \frac{1}{t} \int_0^t \langle w_s,a\rangle\, ds\leq 0,\]
which is the desired result by definition of $w_t$.
\end{proof}

We are now in position to finish the proof of Theorem~\ref{th:main-nonlin}.

\begin{proof}[Proof of Theorem~\ref{th:main-nonlin}~(\ref{th:main-nonlin-slow-weak}).]
Assume that $\rho=\int_\X\frac Qa=1$.
Since $(v_t)_{t\geq0}$ is a family of probability measures, there exists by weak-* compactness a sub-sequence which converges in the vague topology to a positive measure $\mu$ with mass $\langle\mu,\1\rangle\leq1$.
Due to Lemma~\ref{lem:liminf}, this limit must verify $\mu\geq Q/a$, and since the mass of $Q/a$ is 1, the measure $\mu$ is necessarily equal to $Q/a$.
The uniqueness of the limit guarantees that the whole family $(v_t)$ converges to $Q/a$ for the vague topology and, since in $\P(\X)$ vague convergence is equivalent to narrow convergence, the proof is complete.
\end{proof}

\begin{proof}[Proof of Theorem~\ref{th:main-nonlin}~(\ref{th:main-nonlin-conc}).]
Assume that $\rho=\int_\X\frac Qa<1$ and $\lim_{|x|\to\infty}a(x)=+\infty$, and define for all $t\geq0$ the probability measure
\[\bar v_t=\frac 1t \int_0^t v_s\,ds.\]
Lemma~\ref{lem:limsup} and the fact that $a$ tends to $+\infty$ at infinity guarantee that the family $(\bar v_t)_{t\geq0}$ is tight.
Prokhorov's theorem then ensures the existence of a sub-sequence which converges in the narrow topology to a probability measure $\mu$, which must satisfy $\mu\geq Q/a$ by virtue of Lemma~\ref{lem:liminf}.
Since Lemma~\ref{lem:limsup} also ensures that $\langle \mu,a\rangle\leq \langle Q/a,a\rangle$, and $a(x)>0$ for all $x\neq0$,
we deduce that $\mu-Q/a$ must be supported by $\{0\}$.
This means that $\mu=\rho\delta_0+Q/a$, and the proof is complete since the whole family $(\bar v_t)$ must converge to this unique limit.
\end{proof}

\section*{Acknowledgments}
The authors are grateful to Jérôme Coville and Tristan Roget for discussion on the subject.
They also thank the anonymous reviewer for his useful comments, corrections, and suggestions that improved the paper.
The authors have been supported by the ANR project NOLO (ANR-20-CE40-0015), funded by the French Ministry of Research.
B.C. also received the support of the Chair ``Mod\'elisation Math\'ematique et Biodiversit\'e'' of VEOLIA-Ecole Polytechni\-que-MnHn-FX.

%\bibliographystyle{abbrv}
%\bibliography{ref.bib}

\end{document}